\documentclass[10pt,oneside,leqno]{amsart}
\usepackage{amsxtra}
\usepackage{amsopn}
\usepackage{amsmath,amsthm,amssymb}
\usepackage{amscd}
\usepackage{amsfonts}
\usepackage{latexsym}
\usepackage{verbatim}

\theoremstyle{plain}
\newtheorem{theorem}{Theorem}[section]
\newtheorem{definition}[theorem]{Definition}
\newtheorem{lemma}[theorem]{Lemma}
\newtheorem{proposition}[theorem]{Proposition}
\newtheorem{corollary}[theorem]{Corollary}
\newtheorem{remark}[theorem]{Remark}
\newtheorem{example}[theorem]{Example}
\newtheorem{question}[theorem]{QUESTION}
\newtheorem{remark-question}[section]{Remark-Question}
\newtheorem{conjecture}[section]{Conjecture}

\newcommand\R{{\mathbb R}}

\newcommand\SU{{\rm SU}}

\newcommand\db{{\bar{\partial}}}

\begin{document}
\title[Strong K\"ahler with torsion  structures]{Strong K\"ahler with torsion  structures from almost contact manifolds}

\author{Marisa  Fern\'andez}
\address[Fern\'andez]{Universidad del Pa\'{\i}s Vasco\\
Facultad de Ciencia y Tecnolog\'{\i}a, Departamento de Matem\'aticas\\
Apartado 644, 48080 Bilbao\\ Spain} \email{marisa.fernandez@ehu.es}

\author{Anna Fino}
\address[Fino]{Dipartimento di Matematica\\
 Universit\`a di Torino\\
Via Carlo Alberto 10\\
10123 Torino, Italy} \email{annamaria.fino@unito.it}

\author{Luis Ugarte}
\address[Ugarte]{Departamento de Matem\'aticas\,-\,I.U.M.A.\\
Universidad de Zaragoza\\
Campus Plaza San Francisco\\
50009 Zaragoza, Spain} \email{ugarte@unizar.es}

\author{Raquel Villacampa}
\address[Villacampa]{Centro Universitario de la Defensa\\
     Academia General Militar\\
     Crta. de Huesca s/n\\
     50090 Zaragoza, Spain.} \email{raquelvg@unizar.es}
\maketitle

\begin{abstract}
For an almost contact metric
manifold $N$, we  find conditions for which either  the total space  of
an $S^1$-bundle over $N$ or the  Riemannian cone over $N$  admits  a
strong K\"ahler with torsion  (SKT) structure. In this way we construct new $6$-dimensional SKT manifolds.  Moreover,  we study the geometric structure induced on a
hypersurface of an SKT manifold, and use such structures to
construct new SKT manifolds via appropriate evolution equations. Hyper-K\"ahler with torsion (HKT) structures on the total space of an $S^1$-bundle over   manifolds  with three almost contact structures are also studied.
\end{abstract}


\section{Introduction}

 On any Hermitian manifold  $(M^{2n}, J, h)$ there exists a unique Hermitian connection
 $\nabla^B$ with totally skew-symmetric torsion,  called in the literature as
 Bismut connection~\cite{Bi}.  The torsion $3$-form $h(X,T^B(Y,Z))$ of $\nabla^B$
 can  be identified with the $3$-form
$$
- J dF (\cdot, \cdot, \cdot) = -  dF (J \cdot, J \cdot, J \cdot),
$$
where  $F (\cdot, \cdot) = h ( \cdot, J \cdot)$ is the fundamental
$2$-form  associated to the Hermitian structure $(J, h)$.

 Hermitian structures with closed $JdF$ are called {\em strong K\"ahler with torsion} (shortly  {\em SKT})
or also {\em pluriclosed}~\cite{Eg}.  Since $\partial\db$ acts as
$\frac12 dJd$ on forms of bidegree $(1,1)$, the latter condition is
equivalent to $\partial\db F=0$.   SKT structures have been recently
studied by many authors and they have also applications in type II
string theory and in 2-dimensional supersymmetric
$\sigma$-models~\cite{GHR, Str, IP}.

The  class of SKT metrics includes of course the K\"ahler metrics,
but as in~\cite{FPS}  we are interested on non-K\"ahler geometry, so
for SKT metrics we will mean Hermitian metrics $h$  such that  its
fundamental $2$-form $F$ is $\partial\db$-closed but not $d$-closed.

Gauduchon in~\cite{Ga} showed that  on a compact complex surface an
SKT metric can be found  in the conformal class of any given
Hermitian metric, but in higher dimensions the situation is more
complicated.

SKT structures on $6$-dimensional nilmanifolds, i.e. on compact
quotients of nilpotent Lie groups by discrete subgroups, were
classified in~\cite{FPS,U}. Simply-connected  examples of
$6$-dimensional  SKT manifolds have been   found in~\cite{GGP} by
using torus bundles and recently Swann in~\cite{Sw}  has reproduced
them via  the twist construction, by  extending them to higher
dimensions, and finding new  other compact simply-connected SKT
manifolds. Moreover,  in~\cite{FT2} it has been showed that the SKT
condition is preserved by the blow-up construction.

The odd dimensional analog of  Hermitian structures are  given by
normal  almost contact metric structures.  Indeed,  on the product
$N^{2n + 1} \times \R$ of a  $(2n + 1)$-dimensional almost contact
metric manifold $N^{2n + 1}$ by  the real line $\R$
 it is possible to define a natural almost complex structure, which is integrable
 if and only if the almost contact metric structure on $N^{2n + 1}$ is normal~\cite{SH}.
 More in general, it is possible   to construct Hermitian  manifolds  starting from an
 almost contact  metric manifold $N^{2n + 1}$  by considering a principal fibre bundle $P$
 with  base space $N^{2n + 1}$ and structural group $S^1$, i.e.
 an $S^1$-bundle over $N^{2n+1}$ (see~\cite{Og}). Indeed, in~\cite{Og} by  using
 the almost contact metric structure    on $N^{2n + 1}$  and the connection
$1$-form  $\theta$,  Ogawa  constructed  an almost Hermitian
structure $(J, h)$ on $P$ and found conditions for which  $J$ is
integrable and $(J, h)$ is K\"ahler.

In  Section \ref{S1bundlessection} we determine conditions for which
in general an $S^1$-bundle over an almost contact  metric $(2n +
1)$-dimensional manifold  $N^{2n + 1}$ is SKT
(Theorem~\ref{non-t-circle-bundle}).  We study  the particular case
when $N^{2n + 1}$ is  quasi-Sasakian, i.e.   it has an  almost
contact metric structure  for which the  fundamental form is closed
(Corollary~\ref{S1bundle-quasiSasaki}). In  this way we are able to
construct some new $6$-dimensional SKT examples, starting from
$5$-dimensional  quasi-Sasakian  Lie algebras and also from Sasakian
ones.

A Sasakian structure can be also  seen as the analog in odd
dimensions of a K\"ahler structure. Indeed, by~\cite{BG} a
Riemannian manifold  $(N^{2n + 1}, g)$  of odd dimension $2n + 1$
admits a compatible Sasakian structure if and only if the Riemannian
cone $N^{2n + 1} \times \R^+$ is K\"ahler. In
Section~\ref{sectioncones} we study which conditions has to satisfy
the compatible almost contact metric structure on a Riemannian
manifold $(N^{2n + 1}, g)$ in order to the Riemannian cone $N^{2n +
1} \times \R^+$ to be SKT (Theorem~\ref{Riemcone}). An example of an
SKT manifold constructed as Riemannian cone is provided and     the
particular case that the Riemannian cone is $6$-dimensional is
considered in Section~\ref{secconesSKTSU(3)}. This case  is
interesting since one can   impose that
 the SKT structure   is in addition an SKT
SU(3)-structure and one  can find relations with the
SU(2)-structures studied by Conti and Salamon in~\cite{CS}.

In Section~\ref{inducedhyper} we study   the geometric structure
induced naturally on any oriented hypersurface $N^{2n + 1}$  of a
$(2n + 2)$-dimensional  manifold $M^{2n + 2}$  carrying  an SKT
structure and in  Section~\ref{evolutioneq} we
 use such structures to
construct new SKT manifolds via appropriate evolution
equations~\cite{H,CS}, starting from  a $5$-dimensional manifold
endowed with an SU(2)-structure (Theorem~\ref{evolution-eqs-3}).

A good quaternionic analog of K\"ahler geometry is given by  {\em
hyper-K\"ahler with torsion}  (shortly {\em HKT}) geometry. An HKT
manifold is a hyper-Hermitian manifold $(M^{4n}, J_1, J_2, J_3, h)$
admitting a hyper-Hermitian connection with totally skew-symmetric
torsion, i.e. for which the three Bismut connections associated to
the three Hermitian structures $(J_r, h)$, $r =1,2,3$,  coincide.
This geometry was introduced by Howe and Papadopoulos \cite{HP} and
later studied for instance  in~\cite{GP, FG, BDV, BF, Sw}.

A particular interesting case is  when  the torsion $3$-form of such
hyper-Hermitian connection is closed. In this case the HKT manifold
is  called {\em {strong}}.

In the last section we  find conditions for which an $S^1$-bundle
over a $(4n + 3)$-dimensional manifold endowed with three almost
contact metric structures is HKT and in particular when it is strong
HKT (Theorem~\ref{HKTcircle-bundle}).

\section{SKT structures arising from $S^1$-bundles} \label{S1bundlessection}

  Consider a $(2n + 1)$-manifold $N^{2n + 1}$ with an almost
contact metric structure  $(I, \xi, \eta, g)$, that is, $I$ is a
tensor field of type $(1,1)$, $\xi$ is a vector field, $\eta$ is a
$1$-form and $g$ is a Riemannian metric on $N^{2n + 1}$ satisfying
the following conditions:
$$
I^2 = -Id + \eta\otimes\xi,  \quad  \eta(\xi) = 1,  \quad
g(IU,IV)=g(U,V)-\eta(U) \eta(V),
$$
for any vector fields $U$, $V$ on $N^{2n + 1}$. Denote by $\omega$
the fundamental $2$-form
 of $( I, \xi, \eta, g)$, i.e. $\omega$ is the $2$-form on $N^{2n + 1}$ given by
 $$
 \omega(.,.) = g(.,I.).
 $$
 Given the tensor field $I$ consider its Nijenhuis torsion $[I, I]$ defined by
 \begin{equation} \label{NijenhuistensorI}
 [I, I] (X, Y) = I^2 [X, Y] + [IX, IY] - I[IX, Y] - I[X, IY].
\end{equation}
 On the product $N^{2n + 1} \times \R$ it is possible to define a natural almost complex structure  $$
 J \left( X, f \frac{d} {dt} \right) = \left (I X + f \xi, - \eta(X) \frac{d}{dt} \right),
 $$
 where $f$ is a ${\mathcal C}^{\infty}$-function on $N^{2n + 1} \times \R$ and $t$ is the coordinate on $\R$.

We recall the following
\begin{definition} \cite{SH} An almost
contact metric structure  $(I, \xi, \eta, g)$ on $N^{2n + 1}$ is
called {\em normal} if  the almost complex structure $J$ on  $N^{2n
+ 1} \times \R$ is integrable, or equivalently if
$$
[I, I] (X, Y)  + 2 d \eta (X, Y) \xi =0,
$$
for any vector fields $X,Y$ on $N^{2n + 1}$.
\end{definition}

By \cite[Lemma 2.1]{Blair} for a normal almost contact metric
structure  $(I, \xi, \eta, g)$, one has that $i_{\xi} d \eta =0$.

\begin{remark} \label{deta}  {\rm The normality of the almost contact structure  implies  also that $I d \eta = d \eta$. Indeed,     we have that $d (\eta  -  i dt) = d \eta$ has no $(0,2)$-part and therefore it has also no $(2,0)$-part  since $d \eta$ is real. Thus $J d \eta =   d  \eta$, but  we have also that $J d \eta = I d \eta$ since $i_{\xi} d \eta =0$. }
\end{remark}

\medskip

We recall that a Hermitian manifold $(M, J, h)$ is SKT if and only if the $3$-form  $J dF$ is closed, where $F$ is the fundamental $2$-form of $(J, h)$. In the paper
we will use the convention that  $J$ acts on
 $r$-forms $\beta$ as  $$(J  \beta) (X_1,  \ldots,  X_r) =   \beta
(J X_1, \ldots,  J X_r),$$ for any vector fields $X_1,\ldots,X_r$.

We now show conditions for which in general an $S^1$-bundle over an
almost contact metric  $(2n+1)$-dimensional manifold is SKT.

Let $(N^{2n+1}, I, \xi, \eta)$ be a $(2n+1)$-dimensional  almost contact manifold, and
let $\Omega$ be a closed $2$-form on $N^{2n+1}$ which represents an
integral cohomology class on $N^{2n+1}$. From the well-known result of
Kobayashi \cite{Kob}, we can consider the circle bundle $S^1
\hookrightarrow P \to N^{2n+1}$, with connection $1$-form $\theta$ on $P$
whose curvature form is $d\theta = \pi^*(\Omega)$, where $\pi: P \to
N^{2n+1}$ is the projection.

 By using the almost contact structure $(I, \xi, \eta)$ and the connection
$1$-form $\theta$, one can define an almost complex structure $J$ on $P$ as
 follows (see \cite{Og}). For  any
right-invariant vector field  $X$ on $P$,
 $J X$  is given by
\begin{equation} \label{acxS1}
\begin{array} {l}
\theta (JX) = - \pi^*( \eta (\pi_* X)),\\
\pi_* (JX) = I (\pi_* X)  + \tilde \theta (X) \xi,
\end{array}
\end{equation}
where $\tilde \theta (X)$ is the unique function on $N^{2n+1}$ such that
\begin{equation} \label{acxS2}
\begin{array} {l}
\pi^* \tilde  \theta (X) =  \theta (X).
\end{array}
\end{equation}
The above definition can be extended  to arbitrary vector fields $X$
on $P$, since $X$ can be written in the form
$$
X = \sum_j f_j X_j,
$$
with $f_j$ smooth functions on $P$ and $X_j$ right-invariant vector
fields. Then $JX = \sum_j f_j  J X_j$.

In \cite{Og} it has been showed that if  $(N^{2n+1}, I, \xi, \eta)$
is normal, then the almost complex structure $J$ on $P$ defined by
\eqref{acxS1} is integrable if and only if $d \theta$ is
$J$-invariant, that is,
$$
J(d  \theta) = d  \theta,
$$
or equivalently
$$
d  \theta (J X, Y) + d  \theta (X, JY) = 0,
$$
for any vector fields  $X, Y$ on $P$, i.e. $d  \theta$ is a complex
$2$-form on $P$ having bidegree $(1,1)$ with respect to $J$.

\bigskip

 In terms of the $2$-form $\Omega$ whose lifting to $P$
is the curvature of the circle bundle $S^1 \hookrightarrow P \to N^{2n+1}$,
the previous condition means that $\Omega$ is $I$-invariant, i.e.
$I(\Omega)=\Omega$, and therefore $i_{\xi} \Omega = 0$.

If $\{ e^1, \ldots, e^{2n}, \eta \}$ is an adapted coframe on a
neighborhood $U$ on $N^{2n + 1}$, i.e. such that
$$
Ie^{2j-1} = - e^{2j}, \quad   Ie^{2j} =  e^{2j-1},  \quad 1\leq j \leq n,
$$
then we can take $\{ \pi^* e^1, \ldots,
\pi^* e^{2n}, \pi^* \eta, \theta \}$ as a
 coframe in $\pi^{- 1} (U)$. By  using the coframe  $\{ \pi^* e^1, \ldots, \pi^* e^{2n} \}$,
 we may write
$$
d \theta =  \pi^* \alpha + \pi^*\beta \wedge  \pi^* \eta,
$$
where $\alpha$ is a $2$-form  in $\bigwedge^2 < e^1, \ldots,
e^{2n}>$ and $\beta \in \bigwedge^1  < e^1, \ldots, e^{2n}>$.

Next, suppose that $N^{2n+1}$ has a normal almost contact metric
structure $(I, \xi, \eta, g)$. We consider a principal $S^1$-bundle
$P$ with base space $N^{2n+1}$ and connection $1$-form $\theta$, and
endow  $P$  with the almost complex structure $J$ (associated to
$\theta$) defined by~\eqref{acxS1}.  Since $N^{2n+1}$ has a
Riemannian metric $g$, a Riemannian metric $h$ on $P$ compatible
with $J$ (see \cite{Og})
 is given by
\begin{equation} \label{acxS2-1}
\begin{array} {l}
 h(X, Y) = \pi^* g( \pi_* X,  \pi_* Y) +  \theta(X)  \theta(Y),
\end{array}
\end{equation}
for any  right-invariant vector fields $X, Y$. The above definition
can be extended  to any vector field on $P$.

\begin{theorem}\label{non-t-circle-bundle}
Let $(N^{2n+1}, I, \xi, \eta, g)$ be a $(2n+1)$-dimensional  almost
contact  metric manifold and let $\Omega$ be a closed $2$-form on
$N^{2n+1}$ which represents an integral cohomology class. Consider
the circle bundle $S^1 \hookrightarrow P \to N^{2n+1}$ with
connection $1$-form $\theta$ whose curvature form is $d\theta =
\pi^*(\Omega)$, where $\pi: P \to N^{2n+1}$ is the projection. Then,
the almost Hermitian structure $(J, h)$ on $P$, defined by
 \eqref{acxS1}  and  \eqref{acxS2-1}, is {\rm SKT} if and only  if
 $( I, \xi, \eta, g)$ is normal,  $d\theta$ is $J$-invariant and such that
 \begin{equation} \label{acxS2-2}
\begin{array} {l}
d(\pi^* (I(i_{\xi} d \omega))) =0,\\
d(\pi^* (I(d \omega)- d \eta \wedge \eta)) = \left(- \pi^* (I(i_{\xi} d \omega))+
\pi^* \Omega\right)\wedge  \pi^* \Omega,
\end{array}
\end{equation}
 where $\omega$ denotes the  fundamental form of the almost contact metric  structure  $(I, \xi, \eta, g)$.
\end{theorem}

\begin{proof}
As we mentioned previously, a result of Ogawa \cite{Og}  asserts that the
almost complex structure $J$ is integrable if and only if $(I, \xi, \eta, g)$
is normal and $J(d\theta) = d\theta$. Thus $(J, h)$ is SKT if and
only if the $3$-form $J dF$ is closed. By using the first equality
of \eqref{acxS1}, we have that  the fundamental  $2$-form $F$ on $P$ is
$$
\begin{array}{lcl}
F (X, Y) &=& h(X, JY) =  \pi^*g ( \pi_* X,  \pi_* JY) +  \theta (X)
\theta(JY) \\[4pt]
&=& \pi^*g ( \pi_* X,  \pi_* JY) -  \theta (X) \pi^* \eta (\pi_*
Y).
\end{array}
$$
Therefore, taking into account that we are working with a circle
bundle, and so its fibre is $1$-dimensional, we have
$$
F = \pi^* \omega + \pi^* \eta \wedge \theta.
$$
Thus,
$$
d F = \pi^*(d \omega) + \pi^*(d \eta) \wedge  \theta - \pi^* \eta
\wedge d  \theta,
$$
and
\begin{equation} \label{acxS3}
\begin{array}{l}
JdF  =  J(\pi^*(d \omega)) - J( \pi^*(d \eta)) \wedge  \pi^* \eta -
\theta \wedge d  \theta,
\end{array}
\end{equation}
since $J(\pi^* \eta)= \theta$ and $J$ is integrable, so $J(d
\theta)=d  \theta$.

Moreover, we have
\begin{equation} \label{acxS4}
\begin{array} {l}
J(\pi^*(d \omega))= \pi^*(I(d \omega)) + \pi^* ( I(i_{\xi} d \omega)
) \wedge  \theta.
\end{array}
\end{equation}

Indeed, locally and in terms of the adapted basis $\{ e^1, \ldots,e^{2n+1} \}$  such that
$$
Ie^{2j-1} = - e^{2j}, \quad  1\leq j \leq n,  \quad I e^{2n+1} =0,
\quad \eta = e^{2n+1},
$$
we can write
$$
d \omega =  \alpha + \beta \wedge \eta,
$$
where the local forms $\alpha \in \Lambda^3 <e^1, \ldots, e^{2n}>$
and $\beta \in \Lambda^2 <e^1, \ldots, e^{2n}>$ are generated only
by $e^1, \ldots, e^{2n}$. Furthermore,  we have
$$
 I \alpha = I (d \omega),  \quad  \beta = i_{\xi} d \omega.
 $$
Thus,
$$
J(\pi^*(d \omega))= J( \pi^*(\alpha)) +  J(\pi^*(i_{\xi} d \omega))
\wedge \theta.
$$
Now, by using~\eqref{acxS1}  and~\eqref{acxS2}, we see that $J(
\pi^*(\alpha)) = \pi^*(I \alpha)$ and $J(\pi^*(i_{\xi} d \omega)) =
\pi^*(I (i_{\xi} d \omega))$, which proves~\eqref{acxS4}. As a
consequence of Remark~\ref{deta}  we have
\begin{equation} \label{acxS5}
\begin{array}{l}
J(\pi^*(d \eta))= \pi^*(I(d \eta))- \pi^*(I(i_{\xi} d \eta)) \wedge
\theta =  \pi^*(d \eta),
\end{array}
\end{equation}
since $i_{\xi} d \eta =0$ and $I d \eta = d \eta$.

By using~\eqref{acxS4} and~\eqref{acxS5}  we get
\begin{equation} \label{expressionJdF}
J dF =  \pi^*(I(d \omega)) + \pi^*(I(i_{\xi} d \omega))\wedge
\theta - \pi^*(d \eta) \wedge \pi^* \eta  - \theta \wedge d  \theta.
\end{equation}
Therefore
$$
\begin{array}{lll}
d (JdF)& =&  d\left( \pi^*(I(d \omega))\right) + d(\pi^*\{I(i_{\xi}
d \omega)\}) \wedge  \theta
+\pi^*(I(i_{\xi} d \omega)) \wedge  d \theta\\[4pt]
&&- d(\pi^*(d \eta))\wedge \pi^* \eta - \pi^*(d \eta)\wedge d \pi^*
\eta - d \theta \wedge d  \theta.
\end{array}
$$
Consequently, $d (JdF) = 0$ if and only if
$$
d(\pi^*(I(i_{\xi} d \omega))) =0,
$$
and
$$
d( \pi^*(I(d \omega)- d \eta \wedge \eta)) = \left(\pi^*(- I(i_{\xi}
d \omega)) + d \theta\right)\wedge d \theta,
$$
which completes the proof.
\end{proof}

We recall that an almost  contact  metric manifold  $(N^{2n  + 1},
I, \xi, \eta, g)$ is {\em quasi-Sasakian} if it  is  normal and its
fundamental form $\omega$  is closed.  If, in particular,
$d\eta=\alpha\,\omega$, then the almost contact metric  structure is
called {\em $\alpha$-Sasakian}.  When $\alpha=-2$, the structure is
said to be {\em Sasakian}.

By \cite[Theorem 8.2]{FI2} an almost contact metric manifold
$(N^{2n+1}, I, \xi, \eta, g)$ admits a  connection $\nabla^c$
preserving the almost contact metric structure and with totally
skew-symmetric torsion tensor  if and only if the Nijenhuis tensor
of $I$, given by \eqref{NijenhuistensorI},   is skew-symmetric and
$\xi$ is a Killing vector field. Moreover, this connection is
unique.

Then, in particular on any quasi-Sasakian manifold $(N^{2n + 1}, I,
\xi, \eta, g)$ there exists a unique connection $\nabla^c$   with
totally skew-symmetric torsion such that
$$
\nabla^c I =0,  \quad \nabla^c g =0, \quad  \nabla^c  \eta =0.
$$
Such  connection $\nabla^c$  is uniquely determined  by
\begin{equation} \label{almostcontactconnection}
g (\nabla^c_X  Y, Z) = g(\nabla^g_X Y, Z)  + \frac{1}{2} (d \eta \wedge \eta)
(X, Y, Z),
\end{equation}
where $\nabla^g$ denotes the Levi-Civita connection and $\frac{1}{2} (d \eta \wedge \eta)$ is the torsion $3$-form of $\nabla^c$.

\begin{corollary} \label{S1bundle-quasiSasaki}
Let $(N^{2n+1}, I, \xi, \eta, g)$ be a quasi-Sasakian
$(2n+1)$-manifold and let $\Omega$ be a closed $2$-form on
$N^{2n+1}$ which represents an integral cohomology class. Consider
the circle bundle $S^1 \hookrightarrow P \to N^{2n+1}$ with
connection $1$-form $\theta$ whose curvature form is $d\theta =
\pi^*(\Omega)$, where $\pi: P \to N^{2n+1}$ is the projection. Then,
the almost Hermitian structure $(J, h)$ on $P$, defined by
 \eqref{acxS1}  and  \eqref{acxS2-1}, is {\rm SKT} if and only  if $\Omega$ is $I$-invariant, $i_{\xi} \Omega=0$ and
 \begin{equation} \label{bundlequasiSasak}
\begin{array} {l}
d \eta \wedge d\eta =  - \Omega \wedge \Omega.
\end{array}
\end{equation}
Moreover,  the Bismut connection $\nabla^B$  of $(J, h)$  on $P$  and the  connection $\nabla^c$ on $N$  given by \eqref{almostcontactconnection} are related by
\begin{equation} \label{bismutnablac}
h (\nabla^B_X  Y,  Z)  =  \pi^* g( \nabla^c_{\pi_* X} \pi_* Y, \pi_* Z),
\end{equation}
for any vector fields $X, Y, Z \in {\mbox {Ker}} \,  \theta$.
\end{corollary}

\begin{proof} Since $d \omega = 0$, if we impose the SKT condition,   by using the previous theorem,  we get the equation \eqref{bundlequasiSasak}.

The Bismut connection  $\nabla^B$ associated to the Hermitian
structure $(J, h)$ on $P$  is given by:
\begin{equation} \label{bismutconnection}
h (\nabla^B_X Y, Z) = h (\nabla^h_X Y, Z)  - \frac12\, dF (JX, JY,
JZ),
\end{equation}
for any vector fields $X, Y, Z$ on $P$,
where $\nabla^h$ is the Levi-Civita  connection associated to $h$.  Then,  for any $X, Y, Z$ in the kernel of $\theta$ we have
$$
h (\nabla^B_X  Y,  Z) =  \pi^* g(\nabla^h_{X} Y,  Z)  + \frac 12 (\pi^* (d \eta)
\wedge \pi^* \eta)  (X,  Y,  Z) .
$$
By \cite [Lemma 3] {Og}  and the definition of $\nabla^c$ we get
$$
h (\nabla^B_X  Y,  Z)
= \pi^* g(\nabla^g_{\pi_* X} \pi_* Y, \pi_* Z) +   \frac {1}{2} (\pi^ *(d \eta)
\wedge \pi^* \eta)  (X,   Y, Z) = \pi^* g( \nabla^c_{\pi_* X} \pi_* Y, \pi_* Z),
$$
for any $X, Y, Z$ in the kernel of $\theta$.
\end{proof}

\begin{remark}{\em
If the structure $(I, \xi, \eta, g)$ is $\alpha$-Sasakian, equation
\eqref{bundlequasiSasak} reads as $$\Omega \wedge
\Omega=-\alpha^2\,\omega \wedge \omega.$$}
\end{remark}

In the case of a trivial $S^1$-bundle,  i.e. by considering the
natural almost Hermitian structure on the product $N^{2n + 1} \times
\R$, we get the following

\begin{corollary} \label{product}  Let $(N^{2n+1}, I, \xi, \eta, g)$ be a $(2n+1)$-dimensional
almost contact  metric manifold.  Consider  on the  product $N^{2n +
1} \times \R$ the  almost  complex structure~$J$ given by
$$
JX = I X,   \quad X \in {\mbox {Ker}} \,  \eta, \quad J \xi = -
\frac {d} {dt},
$$
and  the product metric $h = g + (dt)^2$. The  Hermitian structure
$(J, h)$  is {\rm SKT} if and only  if $(I, \xi, \eta, g)$ is normal
and such that
$$
 d ( I (d \omega) )= d (d \eta \wedge \eta), \quad d  ( I (i_{\xi} d \omega))  = 0,
$$
 where $\omega$  denotes the  fundamental $2$-form of the almost contact metric structure  $(I, \xi, \eta, g)$.
\end{corollary}

As a consequence of previous results  we get

\begin{corollary} \label{quasiSasaki} Let $(N^{2n + 1}, I, \xi, \eta, g)$ be a  $(2n + 1)$-dimensional quasi-Sasakian  manifold  such that
$d \eta \wedge d \eta =0$.  Then, the  Hermitian structure $(J, h)$
on $N^{2n + 1} \times \R$  is {\rm SKT}. Moreover, its Bismut
connection $\nabla^B$ coincides with the unique connection
$\nabla^c$ on $N^{2n + 1}$ given by \eqref{almostcontactconnection}.

\end{corollary}

\begin{proof} In this case, since $d \omega =0$   we get
$$
d (J dF) = -  d ( d \eta \wedge \eta).
$$
Moreover, by using \eqref{bismutnablac}
$$
h (\nabla^B_X Y, Z) = g (\nabla^c_X Y, Z),
$$
for any vector fields $X, Y, Z$ on $N^{2n + 1}$.
 \end{proof}

\subsection{Examples}

We will start presenting three examples of quasi-Sasakian Lie
algebras satisfying the condition $d \eta \wedge d \eta =0$. By
applying  Corollary \ref{quasiSasaki} one gets an SKT structure on
the product of the corresponding simply-connected Lie group by~$\R$.

\begin{example} \label{ex-prod} {\rm Let  ${\mathfrak s}$  be the $5$-dimensional
Lie  algebra with structure equations
$$\begin{cases}
\begin{array}{l}
 d e^1 =
e^{13} + e^{23} + e^{25} - e^{34} + e^{35},\\
d e^2 = 2  e^{12} - 2 e^{13} + e^{14}- e^{15}-  e^{24} + e^{34} + e^{45} ,\\
 d e^3 = - e^{12}  + e^{13}  + e^{14} -e^{15}  + 2 e^{24}  - 2 e^{34}  + e^{45},\\
 d e^4 = - e^{12} - e^{23} + e^{24} - e^{25} - e^{35},\\
 d e^5 = e^{12} - e^{13} - e^{24} + e^{34},
 \end{array}\end{cases}
$$
where by $e^{ij}$ we denote $e^i \wedge e^j$.

Consider on ${\mathfrak s}$  the quasi-Sasakian structure $(I, \xi,
\eta, g)$ given by
\begin{equation}\label{quasiSask-st}\eta = e^5,  \quad  I e^1 =  - e^2, \quad  I e^3 = - e^4, \quad   \omega = -e^{12} - e^{34},\quad g=\sum_{j=1}^5(e^j)^2.\end{equation}
We have that the  above quasi-Sasakian  structure satisfies the
condition  $d (d\eta\wedge\eta)=~0$.

The Lie algebra ${\mathfrak s}$  is $2$-step solvable since the
commutator
$$
{\mathfrak s}^1 = [{\mathfrak s}, {\mathfrak s}] = \R< e_1 - e_4,\,
e_2 + e_3,\, e_1 - e_2 + 2 e_3 - e_5>$$ is abelian,  where $\{ e_1,
\ldots, e_5 \}$ denotes the dual basis of  $\{ e^1, \ldots, e^5 \}$.
Moreover ${\mathfrak s}$  has  trivial center,  it  is irreducible
and non unimodular, since we have that the trace of $ad_{e_1}$ is
equal to $-3$. }
\end{example}

\begin{example}\label{ex-prod2} {\rm  Consider  the family of   $2$-step solvable Lie algebras ${\mathfrak s}_a$, $a \in \R - \{ 0 \}$,     given by
$$\begin{cases}
\begin{array}{l}
de^1\,=\,a\,e^{23} + 3\,e^{25},\\
de^2\,=\,-a\,e^{13} -3\,e^{15},\\
de^3\,=\,a\,e^{34},\\
de^4\,=\,0,\\
de^5\,=\,-\frac{a^2}{3}\,e^{34}.
\end{array}\end{cases}
$$
The almost contact  metric structure  $(I, \xi, \eta, g)$  given
by~\eqref{quasiSask-st}   is quasi-Sasakian and satisfies  the
condition $d \eta \wedge d \eta =0$. Moreover, the second cohomology
group of  ${\mathfrak s}_a$  is generated by $e^{12}$ and $e^{45}$.
}

\end{example}

\begin{example}\label{ex-prod3}  {\rm Another example of  family of quasi-Sasakian Lie algebras  satisfying the condition $d \eta \wedge d \eta =0$ is ${\mathfrak g}_b$, $b \in \R - \{ 0 \}$,  with structure equations
 $$
 \begin{cases}
 \begin{array}{l}
de^1\,=\,b\,(e^{13} + e^{14} - e^{23}+e^{24}) + e^{25},\\
de^2\,=\,b\,(-e^{13} + e^{14} - e^{23}-e^{24}) - e^{15},\\
de^3\,=\,2\,e^{45},\\
de^4\,=\,-2\,e^{35},\\
de^5\,=\,-4b^2\,e^{34},
\end{array}
\end{cases}
$$
and endowed with the quasi-Sasakian structure given by \eqref{quasiSask-st}.
The second cohomology group  of ${\mathfrak g}_b$ is generated by $e^{12}$. The Lie algebras ${\mathfrak g}_b$ are not solvable since  for  the commutator  we have $[{\mathfrak g}_b, {\mathfrak g}_b]={\mathfrak g}_b$. }
\end{example}

 The Lie groups underlying  examples \ref{ex-prod2}  and \ref{ex-prod3}  satisfy also the conditions of Corollary
 \ref{S1bundle-quasiSasaki}  with $\Omega \wedge \Omega=0$  just
by considering  as connection $1$-form the $1$-form $e^6$ such that
$de^6=\lambda e^{12}$  and then $\Omega = \lambda e^{12}$.  With
this expression of $de^6$ we have that:  $d^2e^6=0,\, J(de^6)=de^6$
and $de^6\wedge de^6=0$, and therefore equation
\eqref{bundlequasiSasak} is satisfied. Observe that $\lambda=0$
provides  examples   of trivial $S^1$-bundles.

We can recover also one of the $6$-dimensional nilmanifolds found in \cite{FPS}.

\begin{example}\label{heisenberg}
{\rm  Consider the $5$-dimensional  nilpotent Lie algebra with structure equations
$$
\left \{ \begin{array}{l}
d e^j =0, \quad j = 1, \ldots, 4,\\
d e^5 = e^{12} + e^{34},
\end{array} \right.
$$
and endowed with the quasi-Sasakian structure given
by~\eqref{quasiSask-st}. If we  consider the  closed $2$-form
 $\Omega=e^{13} + e^{24}$ and we  apply  Corollary~\ref{S1bundle-quasiSasaki} we have that there exists a non trivial $S^1$-bundle over the corresponding $5$-dimensional nilmanifold. Moreover, since  $ d e^5 \wedge d e^5 = - \Omega   \wedge \Omega \neq 0$, the total space of this $S^1$-bundle is  an  SKT nilmanifold. More precisely, according to the classification
given in~\cite{FPS} (see also \cite{U}),  the nilmanifold  is the one with underlying Lie algebra
isomorphic to ${\mathfrak h}_3 \oplus {\mathfrak h}_3$,  where by ${\mathfrak h}_3$ we denote the real $3$-dimensional Heisenberg  Lie algebra. }
\end{example}

Since the starting Lie algebra in
Example~\ref{heisenberg} is Sasakian, it is natural to start with
other $5$-dimensional Sasakian Lie algebras to construct new SKT
structures in dimension 6. A classification of $5$-dimensional Sasakian  Lie algebras was  obtained
in~\cite{AFV}.

\begin{example}\label{k3}
{\rm Consider the 5-dimensional Lie algebra ${\frak
k}_3$ with structure equations
$$
\left\{ \begin{array}{l}
d e^j =0, \quad j = 1, 4,\\
d e^2 = - e^{13},\\
d e^3 = e^{12},\\
d e^5 = \lambda\, e^{14}+\mu\, e^{23},
\end{array}
\right.
$$
where $\lambda,\mu<0$. By~\cite{AFV}   ${\frak k}_3$ admits the
Sasakian  structure given by
$$
\begin{array}{l}
I e^1 =  e^4, \quad I e^2 = e^3, \quad \eta = e^5,\\[4pt]
g = -\frac{\lambda}{2}\, (e_1)^2 -\frac{\lambda}{2}\,(e_2)^2 -\frac{\mu}{2}\,(e_3)^2-\frac{\mu}{2}\,(e_4)^2+\,(e_5)^2,
\end{array}
$$
and   it is isomorphic to
$\mathbb{R}\ltimes ({\frak h}_3\times \mathbb{R})$. Moreover, by \cite{AFV} the  corresponding  solvable simply-connected Lie group admits a compact quotient by a discrete subgroup.

 Consider on ${\mathfrak k}_3$ the closed  $2$-form
$\Omega=\lambda\,e^{14}-\mu\,e^{23}$. $\Omega$ is  $I$-invariant  and satisfies
$\Omega\wedge\Omega=-2\lambda\mu\,e^{1234}$. Since $e^5$ is the contact form
and $de^5\wedge de^5=2\lambda\mu\, e^{1234}$, again we get by
Corollary~\ref{S1bundle-quasiSasaki} an SKT structure on a non trivial $S^1$-bundle  over the $5$-dimensional solvmanifold. We will denote by $e^6$ the connection $1$-form.

The orthonormal basis
$\{\alpha^1=e^1,\,\alpha^2=e^4,\,\alpha^3=e^2,\,\alpha^4=e^3,\,\alpha^5=e^5,\,\alpha^6=\theta\}$
for the SKT metric satisfies the equations
$$
d\alpha^1=d\alpha^2=0,\quad d\alpha^3=-\alpha^{14},\quad
d\alpha^4=\alpha^{13},$$
$$d\alpha^5=\lambda\,\alpha^{12}+\mu\,\alpha^{34},\quad
d\alpha^6=\lambda\,\alpha^{12}-\mu\,\alpha^{34},
$$
 and  the
complex structure  is given by $J(X_1)=X_2, J(X_3)=X_4, J(X_5)=X_6$,
where $\{X_i\}_{i=1}^6$ denotes the basis dual to $\{\alpha^i\}_{i=1}^6$. Since the fundamental $2$-form is
$F=\alpha^{12}+\alpha^{34}+\alpha^{56}$, one has that the $3$-form  torsion
$T$ of the SKT structure is
$$
T=\lambda\,\alpha^{12}(\alpha^5+\alpha^6)+\mu\,\alpha^{34}(\alpha^5-\alpha^6).
$$
Moreover,
$*T=\lambda\,\alpha^{12}(\alpha^5+\alpha^6)-\mu\,\alpha^{34}(\alpha^5-\alpha^6)$,
where $*$ denotes the Hodge operator of the metric, which implies
that the torsion form is also coclosed.

The only nonzero curvature forms $(\Omega^B)^i_j$ of the Bismut
connection $\nabla^B$ are
$$
(\Omega^B)^1_2=-2\,\lambda^2 \alpha^{12}, \quad\quad
(\Omega^B)^3_4=-2\,\mu^2 \alpha^{34}.
$$
A direct calculation shows that the 1-forms $\alpha^{5}, \alpha^{6}$
and the 2-forms $\alpha^{12},\alpha^{34}$ are parallel with respect
to the Bismut connection, which implies that $\nabla^BT=0$.

Finally, since $\nabla^B \alpha^i\not= 0$ for $i=1,2,3,4$, we
conclude that $Hol(\nabla^B)=U(1)\times U(1)\subset
U(3)$.

}
\end{example}

\section{SKT structures arising from Riemannian cones}\label{sectioncones}

Let $N^{2n + 1}$ be a $(2n + 1)$-dimensional  manifold endowed with
an   almost contact metric structure  $(I, \xi, \eta, g)$ and denote
by $\omega$ its fundamental $2$-form.

 The Riemannian cone of $N^{2n + 1}$ is defined as the manifold $N^{2n + 1} \times \R^+$ equipped with the cone metric:
\begin{equation}\label{metricconeexpr}
h = t^2 g + (dt)^2.
\end{equation}
The cone $N^{2n + 1} \times  \R^+$ has a natural  almost Hermitian
structure defined by
\begin{equation} \label{2formconeexpr}
F = t^2 \omega + t \eta \wedge dt.
\end{equation}
The almost complex structure $J$  on $N^{2n + 1} \times \R^+$ defined by $(F,
h)$  is given by
$$J X = I X,  \,  \,   X \in  {\mbox {Ker}}  \, \eta,  \quad  J \xi =-  t \frac{d}{dt}.
$$ In terms of a local orthonormal  adapted coframe  $\{ e^1, \ldots, e^{2n} \}$ for $g$ such that
\begin{equation} \label{formsexpression}
\omega = - \sum_{j = 1}^{n} e^{2j - 1} \wedge e^{2j},
\end{equation} we have
\begin{equation} \label{almostcxcone}
\begin{array}{l}
J e^{2j - 1} = - e^{2j},  \quad J e^{2j} = e^{2j - 1}, \, j = 1, \ldots, n, \\
   J (t e^{2n + 1}) =  dt,  \quad J (dt) = -  t e^{2n + 1}.
\end{array}
\end{equation}
The almost Hermitian structure $(J, h)$ on  $N^{2n + 1} \times \R^+$
is K\"ahler if and only if the  almost contact metric structure $(I,
\xi, \eta, g)$ on $N^{2n + 1}$ is Sasakian, i.e. a  normal contact
metric structure.

If we impose that the  almost Hermitian structure   $(J, h)$ on
$N^{2n + 1} \times \R^+$ is SKT, we can prove the following

\begin{theorem} \label{Riemcone} Let $(N^{2n + 1},I, \xi, \eta, g)$ be a $(2n + 1)$-dimensional
almost contact metric  manifold. The  almost Hermitian structure
$(J, h)$ on the Riemannian cone $(N^{2n + 1}  \times \R^+, h)$,
given by~\eqref{metricconeexpr} and~\eqref{2formconeexpr},  is SKT
if and only if $(I, \xi, \eta, g)$ is normal and
\begin{equation} \label{SKTCONE}
-4 \eta \wedge  \omega + 2 I (d \omega) -   2  d \eta  \wedge \eta
= d (I (i_{\xi} d \omega)),
\end{equation} where $\omega$ denotes the fundamental $2$-form of the almost contact  metric structure $(I, \xi, \eta, g)$.
\end{theorem}

\begin{proof}  $J$ is integrable if and only if the almost contact metric structure is normal. Now we compute $J dF$. We have that
$$
d F = 2 t dt \wedge \omega + t^2 d \omega + t d \eta \wedge dt,
$$
and
$$
J d F = - 2 t ^2  \eta  \wedge \omega + t^2 J (d \omega)  - t^2   d
\eta \wedge  \eta ,
$$
since
$$
J \omega = \omega, \quad J (dt) = - t \eta, \quad J d \eta = d \eta.
$$
Moreover, with respect to an  adapted basis $\{ e^1, \ldots,e^{2n + 1} \}$ we may prove, in a similar way as in the proof of Theorem \ref{non-t-circle-bundle},
 that
\begin{equation} \label{domega1}
J d \omega = I (d \omega) + I (i_{\xi} d \omega) \wedge J \eta.
\end{equation}

As a consequence we get
$$
J dF = - 2 t^2 \eta \wedge \omega + t^2 I (d \omega) + t dt \wedge
I (i_{\xi} d \omega) - t^2 d \eta \wedge \eta.
$$
Therefore, by imposing $d (J dF) =0$ we obtain  the two equations
$$
\left\{
\begin{array}{l}
-4 \eta \wedge  \omega + 2 I (d \omega) -   2  d \eta \wedge \eta  -d (I (i_{\xi} d \omega))  =0,\\[4pt]
-2d( \eta \wedge \omega) + d( I (d \omega)) - d  (d \eta \wedge
\eta)=0.
\end{array}
\right.
$$
Since the second equation is consequence of the first one,  we have
that the Hermitian structure   $(F, h)$ on the Riemannian  cone
$N^{2n + 1} \times \R^+$  is SKT if and only if the  almost contact
metric structure  $(I, \eta, \xi, g, \omega)$ on $N^{2n + 1}$
satisfies the equation~\eqref{SKTCONE}.
\end{proof}

\begin{remark}{\rm As a consequence of previous theorem we have that, if $n=1$, equation~\eqref{SKTCONE} is satisfied if and only if the 3-dimensional manifold $N$ is Sasakian.  On the other hand, if $n>1$ and
  the  almost contact metric structure on $N^{2n + 1}$ is  quasi-Sasakian (i.e. $d \omega =0$), then the structure has to be Sasakian, i.e. $d \eta = -  2 \omega$.  } \end{remark}

\begin{example} {\rm Consider the $5$-dimensional Lie algebras $\mathfrak g_{a,b,c}$  with structure equations
 $$\begin{cases}
 \begin{array}{l}
 d e^1 = a\,e^{23} + 2\,e^{25} + \left( - \frac{1}{2} ab + \frac{b^3}{ 2 a} + 2 \frac{b} {a} \right)\,e^{34} + b\,e^{45},\\[4pt]
 d e^2 =  - a\,e^{13} - 2\,e^{15} - \frac{1}{2} b c\,e^{34} - b\,e^{35},\\[4pt]
 d e^3 = \left( - \frac{4}{a}- \frac{b^2}{a}  \right)\,e^{34},\\[4pt]
 d e^4 = c\,e^{34},\\[4pt]
  d e^5 = 2\,e^{12} + b\,e^{14}  - b\,e^{23} + (2 + b^2)\,e^{34},
  \end{array}\end{cases}
  $$ where $a,b,c\in\R$ and $a\neq0$,
  endowed with the normal almost contact  metric structure $(I, \xi, \eta, g, \omega)$ with
  $$
  I e^1 = - e^2, \quad I e^3 = - e^4, \quad \eta = e^5,\quad
  \omega = - e^{12} - e^{34}.
  $$
  This structure satisfies \eqref{SKTCONE} and therefore, the Riemannian cones over the corresponding simply-connected Lie groups are  SKT.}
   \end{example}

\section{SKT SU(3)-structures} \label{secconesSKTSU(3)}

Let $(M^{6}, J, h)$ be a $6$-dimensional  almost Hermitian manifold. An $SU(3)$--structure on $M^6$  is determined
by the choice of a  $(3,0)$-form $\Psi = \Psi_+ + i \Psi_-$ of unit norm.
 If $\Psi$ is closed, then the underlying almost complex
structure $J$ is integrable and the manifold is Hermitian. We will denote the $SU(3)$-structure  $(J, h, \Psi)$  simply by  $(F, \Psi)$, where $F$ is the fundamental $2$-form, since from     $F$ and $\Psi$ we can reconstruct the almost Hermitian structure.

We can give the following

\begin{definition}\label{sktSU(3)-structure}
We say that an {\rm SU(3)}-structure $(F, \Psi)$ on $M^6$  is \emph{SKT} if
\begin{equation}\label{sktSU(3)-conditions}
d\Psi=0, \qquad d(JdF)=0,
\end{equation}
where $J$ is the associated complex structure.
\end{definition}

We will see the relation between  SKT $SU(3)$-structures in
dimension $6$ and $SU(2)$-structures in dimension $5$.

First we recall some facts about $\SU(2)$-structures on a
$5$-dimensional manifold.  An {\em $\SU(2)$-structure} on  a $5$-dimensional manifold $N^5$
is an $\SU(2)$-reduction of the principal
bundle of linear frames on $N^5$. By \cite[Proposition 1]{CS},
these structures are in $1:1$ correspondence with quadruplets
$(\eta,\omega_1,\omega_2,\omega_3)$, where $\eta$ is a $1$-form and
$\omega_i$ are $2$-forms on $N^5$ satisfying
$$
\omega_i\wedge\omega_j=\delta_{ij}v, \quad v\wedge\eta\not=0,
$$
for some $4$-form $v$, and
$$
i_X\omega_3=i_Y\omega_1\ \Rightarrow\ \omega_2(X,Y)\ge 0,
$$
where $i_X$ denotes the contraction by $X$. Equivalently, an
$\SU(2)$-structure on $N^5$ can be viewed as the datum of
$(\eta,\omega_1,\Phi)$, where $\eta$ is a $1$-form, $\omega_1$ is a
$2$-form and $\Phi =\omega_2+i\,\omega_3$ is a complex $2$-form such
that
$$
\eta\wedge\omega_1\wedge\omega_1 \neq 0, \quad\quad
\Phi\wedge\Phi=0, \quad\quad \omega_1\wedge\Phi =0, \quad\quad
\Phi\wedge\overline{\Phi} =2\,\omega_1\wedge\omega_1,
$$
and $\Phi$ is of type $(2,0)$ with respect to $\omega_1$.

$\SU(2)$-structures are locally characterized as follows (see
\cite{CS}): If $(\eta,\omega_1,\omega_2,\omega_3)$ is an
$\SU(2)$-structure on a $5$-manifold $N^5$, then locally, there exists
an orthonormal basis of $1$-forms $\{ e^1,\ldots ,e^5\}$ such that
$$
\omega_1 = e^{12}+e^{34},\qquad \omega_2 = e^{13}-e^{24},\qquad
\omega_3 = e^{14}+e^{23},\qquad \eta = e^{5}\,.
$$

We can also consider the local tensor field $I$ given by
$$
I e^1= -e^2,\quad I e^2= e^1,\quad I e^3= -e^4,\quad I e^4=
e^3,\quad I e^5= 0.
$$
This tensor gives rise to a global tensor field of type $(1,1)$ on the manifold
$N^5$ defined by $\omega_1(X,Y)=g(X,IY)$, for any vector fields $X,Y$
 on $N^5$, where $g$ is the Riemannian metric on $N^5$ underlying the $SU(2)$-structure. The tensor field $I$ satisfies
$$
I^2=-Id + \eta\otimes \xi,
$$
where $\xi$ is the vector field on $N^5$ dual to the 1-form $\eta$.

Therefore, given an $SU(2)$-structure
$(\eta,\omega_1,\omega_2,\omega_3)$ we also have an almost contact
metric structure $(I,\xi,\eta, g)$ on the manifold, where $\omega_1$
is the fundamental form.

\begin{remark}
{\rm Notice that we have two more almost contact metric structures
when one considers $\omega_2$ and $\omega_3$ as fundamental
forms.}
\end{remark}

If $N^5$ has an $SU(2)$-structure  $(\eta, \omega_1, \omega_2,
\omega_3)$, the product $N^5\times \R$ has a natural $SU(3)$-structure  given by
\begin{equation} \label{SKTSU3productexp}
\begin{array}{l}
F =  \omega_1 + \eta \wedge dt,\\[4pt]
\Psi =  (\omega_2 + i \omega_3) \wedge (\eta - i dt).
\end{array}
\end{equation}
Moreover, by Corollary  \ref{product} the previous
$SU(3)$-structure is SKT if and only if
\begin{equation} \label{eqSU3sktprod}
\begin{array}{l}
 d ( I (d \omega_1) )= d (d \eta \wedge \eta), \quad d  ( I (i_{\xi} d \omega_1))  = 0,\\[5pt]
 d \omega_2 =  - 3 \,  \omega_3 \wedge \eta, \quad d \omega_3 = 3  \, \omega_2  \wedge \eta.
 \end{array}
\end{equation}
Then we have proved  the following
\begin{theorem} \label{SKTSU(3)product} Let $N^5$ be a $5$-dimensional manifold endowed with an $SU(2)$-structure $(\eta, \omega_1, \omega_2, \omega_3)$. The   $SU(3)$-structure  $(F, \Psi)$, given by \eqref{SKTSU3productexp}, on the  product  $N^5 \times \R$ is SKT if and only if the equations \eqref{eqSU3sktprod} are satisfied.
\end{theorem}

\begin{example} {\rm Consider   on the  $5$-dimensional  Lie algebras, introduced  in Examples   \ref{ex-prod}, \ref{ex-prod2} and \ref{ex-prod3},    the $SU(2)$-structure given by
$$
\omega = \omega_1 = e^{12} + e^{34}, \quad \omega_2 = e^{13} -
e^{24}, \quad \omega_3 = e^{14} + e^{23}.
$$
For the example \ref{ex-prod}  we have:
$$d\omega_2=-2\;\omega_3\wedge\eta - 4(e^{124}-e^{134}),$$ $$d\omega_3=2\;\omega_2\wedge\eta + 4(e^{123}+e^{234}).$$

For the examples \ref{ex-prod2} and \ref{ex-prod3}   we get
$d\omega_2=-3\,\omega_3\wedge\eta$ and
$d\omega_3=3\,\omega_2\wedge\eta$, therefore on the product of the corresponding simply-connected Lie groups by $\R$ one gets an SKT $SU(3)$-structure.}
\end{example}

We will study  the existence of  SKT $SU(3)$-structures on  a
Riemannian cone over a $5$-dimensional manifold $N^5$ endowed with
an $SU(2)$-structure  $(\eta, \omega_1, \omega_2, \omega_3)$.  Then
$N^5$ has  an induced  almost contact metric structure  $( I, \xi,
\eta, g)$ and $\omega_1$ is  its fundamental form.

 The Riemannian  cone $(N^5 \times  \R^+, h)$ of $(N^5, g)$  has a natural $SU(3)$-structure defined by
$$
\begin{array}{l}
F = t^2 \omega_1 + t \eta \wedge dt,\\[4pt]
\Psi = t^2 (\omega_2 + i \omega_3) \wedge (t \eta - i dt).
\end{array}
$$
 In terms of a local orthonormal  coframe  $\{ e^1, \ldots, e^5 \}$ for $g$ such that
$$
\omega_1 = - e^{12} - e^{34},  \quad \omega_2 = - e^{13} + e^{24},
\quad  \omega_3 = - e^{14} - e^{23},  \quad \eta = e^5,
$$
we have that
$$
J e^1 = - e^2,  \quad J e^2 = e^1, \quad  J e^3 = - e^4, \quad  J
e^4 = e^3, \quad J (t e^5) =  dt,  \quad J (dt) = -  t e^5.
$$

We recall that the $SU(3)$-structure $(F, \Psi)$ on $N^5 \times
\R^+$ is integrable if and only if  the $SU(2)$-structure $(\eta,
\omega_1, \omega_2, \omega_3)$ on $N^5$  is Sasaki-Einstein, or
equivalently if and only if
$$
d  \eta =  -2  \, \omega_1, \quad d \omega_2 =  - 3 \,  \omega_3
\wedge \eta, \quad d \omega_3 = 3  \, \omega_2  \wedge \eta.
$$

For the Riemannian cones we can prove the following

\begin{corollary} Let $N^5$ be a $5$-dimensional  manifold endowed with an {\rm SU(2)}-structure $(\eta, \omega_1, \omega_2, \omega_3)$. The {\rm SU(3)}-structure $(F, \Psi)$ on the Riemannian cone $(N^5 \times \R^+, h)$ is {\rm SKT} if and only if
\begin{equation} \label{NEWSKTCONE}
\left \{
\begin{array} {l}
-4 \eta \wedge  \omega_1 + 2 I (d \omega_1) -   2  d \eta  \wedge \eta  = d (I (i_{\xi} d \omega_1)), \\[4pt]
d \omega_2 =  3 \, \omega_3 \wedge \eta,\\[4pt]
 d \omega_3 =  - 3 \,  \omega_2  \wedge \eta.
\end{array}
\right.
\end{equation}

\end{corollary}

\begin{proof} By imposing that $d \Psi =0$ we get  the conditions
$$
d \omega_2 =  -3  \, \omega_3 \wedge \eta, \quad
 d \omega_3 =  3 \,  \omega_2  \wedge \eta.
$$
By imposing   $d (J dF) =0$,  we obtain,   as in the proof of
Theorem  \ref{Riemcone}, the  equation  \eqref{SKTCONE}  for $\omega
= \omega_1$.

\end{proof}

\section{Almost contact metric structure  induced on a hypersurface} \label{inducedhyper}

Here we study the almost contact metric structure induced naturally on any oriented
hypersurface $N^{2n + 1}$ of a $(2n+2)$-manifold $M^{2n+2}$ equipped
with an SKT structure.

Let $f\colon N^{2n + 1} \longrightarrow M^{2n + 2}$ be an oriented
hypersurface of a $(2n + 2)$-dimensional manifold $M^{2n + 2}$
endowed with  an SKT structure $(J, h, F)$ and denote by $\mathbb U$
the unitary normal vector field. It is well known that  $N^{2n + 1}$
inherits an   almost contact metric structure $(I, \xi, \eta, g)$
such that   $\eta$  and the fundamental $2$-form $\omega$ are given
by
\begin{equation}\label{induced-acm-structure}
\eta=-f^*(i_{\mathbb U}F), \quad
 \omega= f^*F,
\end{equation}
where $F$ is the fundamental $2$-form of the almost Hermitian structure (see for instance \cite{Blair2}).

\begin{proposition}\label{hypersurface-anydim}
Let $f\colon N^{2n + 1} \longrightarrow M^{2n + 2}$ be an immersion
of an oriented $(2n + 1)$-dimensional manifold into a $(2n +
2)$-dimensional  Hermitian manifold  $(M^{2n + 2}, J, h)$. If the
 Hermitian structure $(J, h)$  is {\rm SKT}, then the  induced almost contact metric structure $(I, \xi, \eta, g)$  on $N^{2n + 1}$,  with $\eta$ and $\omega$ given
by~\eqref{induced-acm-structure},  satisfies
\begin{equation}\label{SKTcond-anydim}
d \big( Id\omega  - I(f^*(i_{\mathbb U} d F))\wedge \eta \big) =0.
\end{equation}
\end{proposition}

\begin{proof}
We can choose locally an adapted coframe  $\{e^1,\ldots, e^{2n + 2}\}$
for the Hermitian structure  such that the unitary normal vector
field $\mathbb U$ is dual to $e^{2n + 2}$. Since the almost complex
structure $J$ is given in this adapted basis by
$$
\begin{array}{l}
J e^{2j-1}= -e^{2j}, \quad J e^{2j}= e^{2j - 1},  \,  j = 1, \ldots,
n, \\[4pt]
 J e^{2n + 1}= e^{2n + 2},\quad J  e^{2n +2}= -e^{2n + 1},
\end{array}
$$
the tensor field $I$ on $N^{2n + 1}$ satisfies that $I f^* e^i=f^* J e^i$,
 $i=1,\ldots,2n+1$, that is,
$$
I f^* e^{2j-1}= - f^* e^{2j},\quad I f^*e^{2j}= f^* e^{2j-1}, \, j =
1, \ldots, n,  \quad  I f^* e^{2n + 1}= 0.
$$
However, $I f^* e^{2n + 2} =0 \not= f^* e^{2n + 1} = - f^* J e^{2n +
2}$.

Now we compute $f^* J dF$. First we decompose (locally and in terms
of the adapted basis) the differential of $F$ as follows:
$$
dF=\alpha + \beta\wedge e^{2n + 1} + \gamma\wedge e^{2n + 2} +
\mu\wedge e^{2 n + 1} \wedge e^{2n + 2},
$$
where the local forms $\alpha\in \bigwedge^3<e^1, \ldots,  e^{2n}>$,
$\beta,\gamma\in \bigwedge^2<e^1, \ldots,  e^{2n}>$ and $\mu\in
\bigwedge^1<e^1, \ldots,  e^{2n}>$ are generated only by $e^1,
\ldots, e^{2n}$. Then,
$$
JdF=J\alpha + J\beta \wedge e^{2n + 2} - J \gamma \wedge e^{2n + 1}
+ J\mu \wedge e^{2n + 1}\wedge e^{2n + 2}.
$$
Since $f^* e^{2n + 2}=0$ and using that $f^* e^{2n + 1}=\eta$, we
get
$$
f^*JdF= f^*J\alpha - (f^*J \gamma) \wedge \eta.
$$
But $f^* (i_{\mathbb U} dF) = f^* \gamma + f^* \mu \wedge \eta$,
which implies that
$$
I(f^* (i_{\mathbb U} dF))= I f^* \gamma = f^* J \gamma.
$$

On the other hand,
$$
Id\omega= Id f^* F = I f^* dF = I f^*\alpha = f^* J \alpha.
$$

We conclude that
$$
f^* J dF = f^*J\alpha - (f^*J \gamma) \wedge \eta = Id\omega -
I(f^* (i_{\mathbb U} d F))\wedge \eta.
$$
Now, if the  Hermitian structure is SKT, then $J dF$ is closed and
the induced structure satisfies~\eqref{SKTcond-anydim}.
\end{proof}

\begin{remark} \label{remarkhypers}
{\rm  Notice that using that $i_{\mathbb U} d F =
\mathcal{L}_{\mathbb U} F - d i_{\mathbb U} F$ we can write
(\ref{SKTcond-anydim}) as
$$
d \big( Id\omega - I( f^* (\mathcal{L}_{\mathbb U} F) + d
\eta)\wedge \eta \big) =0.
$$
Therefore, if $f^* (\mathcal{L}_{\mathbb U} F) =0$, the induced
almost contact metric structure has to satisfy the equation
$$
d \big( Id\omega - I (d \eta)\wedge \eta \big) =0.
$$
In the case of the product $N^{2n + 1}  \times \R$ the condition
$f^* (\mathcal{L}_{\mathbb U} F) =0$ is satisfied.

In the case of the Riemannian cone we have that
$$\mathcal{L}_{\frac{d}{dt}} F = 2 t \omega + dt \wedge \eta,$$
and therefore we get $f^* (\mathcal{L}_{\frac{d}{dt}} F) = 2 \omega$.

In this way we recover some of the equations obtained in Corollary \ref{product} and in Theorem \ref{Riemcone}.}
\end{remark}

Now  we study the structure induced naturally on any oriented
hypersurface $N^5$ of a $6$-manifold $M^6$ equipped with an SKT
$\SU(3)$-structure.

Let $f\colon N^5 \longrightarrow M^6$ be an oriented hypersurface of
a $6$-manifold $M^6$ endowed with an $\SU(3)$-structure
$(F,\Psi=\Psi_++i\,\Psi_-)$ and denote by $\mathbb U$ the unitary
normal vector field. Then $N^5$ inherits an $\SU(2)$-structure
$(\eta,\omega_1,\omega_2,\omega_3)$ given by
\begin{equation}\label{induced-SU(2)-structure}
\eta=-f^*(i_{\mathbb U}F), \quad
 \omega_1= f^*F, \quad \omega_2=-f^*(i_{\mathbb U} \Psi_-), \quad
 \omega_3=f^*(i_{\mathbb U} \Psi_+). \quad
\end{equation}

As a consequence of Proposition \ref{hypersurface-anydim} we have
the following

\begin{corollary}\label{hypersurface}
Let $f\colon N^5\longrightarrow M^6$ be an immersion of an oriented
$5$-dimensional manifold into a $6$-dimensional manifold with an
$\SU(3)$-structure. If the $\SU(3)$-structure is {\rm SKT}, then the
induced $\SU(2)$-structure on $N^5$ given
by~\eqref{induced-SU(2)-structure} satisfies
\begin{equation}\label{SKTcond-1}
d \big( Id\omega_1 - If^*(i_{\mathbb U} d F)\wedge \eta \big) =0,
\end{equation}
and
\begin{equation}\label{SKTcond-2}
d(\omega_2\wedge\eta)=0,\qquad d(\omega_3\wedge\eta)=0.
\end{equation}
\end{corollary}

\begin{proof}
The equation \eqref{SKTcond-1} follows by Proposition
\ref{hypersurface-anydim} taking $\omega = \omega_1$. We can choose
locally an adapted coframe $\{e^1,\ldots,e^5,e^6\}$ for the
$SU(3)$-structure such that the unitary normal vector field $\mathbb
U$ is dual to $e^6$. From~\eqref{induced-SU(2)-structure} it follows
that $\omega_2\wedge\eta= f^* \Psi_+$ and $\omega_3\wedge\eta= f^*
\Psi_-$. Now, if $\Psi=\Psi_+ + i\, \Psi_-$ is closed then the
induced structure satisfies~\eqref{SKTcond-2}.
\end{proof}

\subsection{A simple example}\label{example}

Consider the $6$-dimensional nilmanifold $M^6$
whose  underlying nilpotent Lie algebra has structure equations
$$
\left \{ \begin{array}{l}
d e^j = 0, j = 1,2,3, 6,\\
d e^4 = e^{12},\\
d e^5 = e^{14},
\end{array}
\right.
$$
and it is endowed with the $SU(3)$-structure given by
$$
F = -e^{14} - e^{26} - e^{53}, \quad \Psi = (e^1 - i e^4) \wedge (e^2 - i e^6) \wedge (e^5 - i e^3).
$$
The oriented hypersurface with normal vector field  dual to $e^2$ is
a $5$-dimensional nilmanifold $N^5$, which has  by~\cite{CS} no
invariant hypo structures,  but  the $\SU(2)$-structure on $N^5$
 \begin{equation} \label{su2structex}
\eta = e^{2},\qquad \omega_1 = -e^{14}-e^{53},\qquad \omega_2 =
-e^{15}-e^{34},\qquad \omega_3 = -e^{13}-e^{45},
\end{equation}
satisfies (\ref{SKTcond-1}) and (\ref{SKTcond-2}).
In section \ref{evolutioneq} we will  show that by using this  $SU(2)$-structure and  appropriate evolution
equations  we can  construct an SKT $SU(3)$-structure  on the product of $N^5$ with an open interval.

\section{SKT evolution equations} \label{evolutioneq}

The goal here is to construct SKT $\SU(3)$-structures by means of
appropriate evolution equations starting from a suitable  $SU(2)$-structure on
a $5$-dimensional manifold, following ideas of \cite{H} and \cite{CS}.

\begin{lemma}\label{evolution-eqs-1}
Let $(\eta(t),\omega_1(t),\omega_2(t),\omega_3(t))$ be a family of
$\SU(2)$-structures on a $5$-dimensional manifold $N^5$, for $t\in (a,b)$. Then, the
$\SU(3)$-structure on $M^6=N^5\times (a,b)$ given by
$$
F=\omega_1(t)+\eta(t)\wedge dt, \quad\quad \Psi = (\omega_2(t) + i
\omega_3(t))\wedge(\eta(t) -  i dt),
$$
satisfies the condition  $d \Psi=0$ if and only if
$(\eta(t),\omega_1(t),\omega_2(t),\omega_3(t))$ is an
$\SU(2)$-structure such that
\begin{equation}\label{totslsystem}
\begin{array}{l}
\hat d(\omega_2(t)\wedge\eta(t))=0,\quad \hat d(\omega_3(t)\wedge\eta(t))=0,\\[5pt]
\partial_t(\omega_2(t)\wedge\eta(t))=-\hat{d} \omega_3(t) ,\quad
\partial_t(\omega_3(t)\wedge\eta(t))=\hat{d} \omega_2(t),
\end{array}
\end{equation}
hold, for any $t$ in the open interval $(a, b)$.
\end{lemma}

Here $\hat{d}$ denotes the exterior differential on $N^5$
 and $d$ the exterior differential on $M^6$.    Now we show which are the additional evolution equations to
add to the last two equations of \eqref{totslsystem} to ensure that $dJdF=0$.

\begin{proposition}\label{evolution-eqs-2}
Let $(\eta(t),\omega_1(t),\omega_2(t),\omega_3(t))$ be a family of
$\SU(2)$-structures on $N^5$, for $t\in (a,b)$. Then, the
$\SU(3)$-structure on $M^6=N^5\times (a,b)$ given by
\begin{equation}\label{SU(3)-on-NxI-bis}
F=\omega_1(t)+\eta(t)\wedge dt, \quad\quad \Psi = (\omega_2(t) + i
\omega_3(t))\wedge(\eta(t)- i dt),
\end{equation}
satisfies that $JdF$ is closed if and only if
$(\eta(t),\omega_1(t),\omega_2(t),\omega_3(t))$ satisfies the
following evolution equations
\begin{equation}\label{evolution-SKT}
\left\{\begin{array}{l}
\hat{d} \Big( I_t  \hat{d} \omega_1(t) - I_t (\partial_t \omega_1(t) + \hat{d} \eta(t))\wedge \eta(t) \Big)=0 ,\\[5pt]
\partial_t \Big( I_t  \hat{d} \omega_1(t) - I_t (\partial_t \omega_1(t) + \hat{d} \eta(t))\wedge \eta(t) \Big) =\\[5pt] -
\hat{d} \Big( I_t ( i_{\xi} \hat{d} \omega_1(t)) - I_t ( i_{\xi}
(\partial_t \omega_1(t) + \hat{d} \eta(t)))\wedge \eta(t) \Big) ,
\end{array} \right.
\end{equation}
where, for each $t\in (a,b)$, $\xi(t)$ denotes the vector field on
$N^5$ dual to $\eta(t)$.
\end{proposition}

\begin{proof}
Since $F=\omega_1(t)+\eta(t)\wedge dt$, we have that
$$
dF=\hat{d}\omega_1 + (\partial_t \omega_1 + \hat{d}\eta)\wedge dt.
$$

Let $\{ e^1(t),\ldots, e^4(t),\eta(t) \}$ be a local adapted basis
for the SU(2)-structure
$(\eta(t),\omega_1(t),\omega_2(t),\omega_3(t))$. Then $\{
e^1(t),\ldots, e^4(t),\eta(t),dt \}$ is an adapted basis for the
SU(3)-structure (\ref{SU(3)-on-NxI-bis}) and $J$ is given by
$$
J e^1(t) = -e^2(t),\  J e^2(t) = e^1(t),\  J e^3(t) = -e^4(t),\  J
e^4(t) = e^3(t),$$ $$J \eta(t) = dt,\  J dt = -\eta(t).
$$
Then, the structures $I_t$ induced on $N^5$ for each $t$ are given
by
$$
I_t  e^1(t) = -e^2(t),\  I_t  e^2(t) = e^1(t),\  I_t  e^3(t) =
-e^4(t),\  I e^4(t) = e^3(t),\  I_t  \eta(t) = 0.
$$

Now, given $\tau(t)\in \Omega^k(N^5)$, $t\in (a,b)$, we can decompose
it locally as
$$
\tau(t)=\alpha(t) + \beta(t)\wedge \eta(t),
$$
where $\alpha(t)\in \bigwedge^k <e^1(t),\ldots,e^4(t)>$ and
$\beta(t)\in \bigwedge^{k-1} <e^1(t),\ldots,e^4(t)>$. Therefore
$$
J \tau(t)= J\alpha(t) + J\beta(t)\wedge J\eta(t)= I_t \alpha(t) +
I_t \beta(t)\wedge dt = I_t \tau(t) - (-1)^k I_t
(i_{\xi(t)}\tau(t))\wedge dt.
$$
Applying this to $JdF$ we get

\smallskip

$JdF= J\hat{d}\omega_1 - J(\partial_t \omega_1 + \hat{d}\eta)\wedge
\eta(t) $

\smallskip

\hskip.8cm $= I_t \hat{d}\omega_1 - I_t (\partial_t \omega_1 +
\hat{d}\eta)\wedge \eta(t) + I_t (i_{\xi(t)} \hat{d}\omega_1)\wedge
dt - I_t \Big( i_{\xi} (\partial_t \omega_1 + \hat{d}
\eta)\Big)\wedge \eta(t) \wedge dt $.

Finally, taking the differential of $JdF$ we get

\smallskip

$dJdF= \hat{d}\Big(I_t \hat{d}\omega_1 - I_t (\partial_t \omega_1 +
\hat{d}\eta)\wedge \eta(t)\Big) + \partial_t\Big(I_t \hat{d}\omega_1
- I_t (\partial_t \omega_1 + \hat{d}\eta)\wedge \eta(t)\Big)\wedge
dt$

\smallskip

\hskip1cm $ +\, \hat{d}\Big[ I_t (i_{\xi(t)} \hat{d}\omega_1) - I_t
\Big( i_{\xi} (\partial_t \omega_1 + \hat{d} \eta)\Big)\wedge
\eta(t) \Big] \wedge dt $.
\end{proof}

\begin{remark}
{\rm Observe that the first equation in (\ref{evolution-SKT}) is
exactly condition (\ref{SKTcond-1}) for $F=\omega_1(t)
+\eta(t)\wedge dt$ (see Remark \ref{remarkhypers}). }
\end{remark}

As a consequence of Lemma~\ref{evolution-eqs-1} and
Proposition~\ref{evolution-eqs-2}, we get

\begin{theorem}\label{evolution-eqs-3}
Let $(\eta(t),\omega_1(t),\omega_2(t),\omega_3(t))$, $t\in (a,b)$,  be a family of
$\SU(2)$-structures on a $5$-dimensional manifold $N^5$, such that
\begin{equation}\label{cond-hypo-type}
\hat d(\omega_2(t)\wedge\eta(t))=0,\quad \hat
d(\omega_3(t)\wedge\eta(t))=0,
\end{equation} for any $t$.
If the following evolution equations
\begin{equation}\label{evolution-SKT-bis}
\left\{\begin{array}{l}
\hat{d} \Big( I_t  \hat{d} \omega_1(t) - I_t (\partial_t \omega_1(t) + \hat{d} \eta(t))\wedge \eta(t) \Big)=0 ,\\[7pt]
\partial_t \Big( I_t  \hat{d} \omega_1(t) - I_t (\partial_t \omega_1(t) + \hat{d} \eta(t))\wedge \eta(t) \Big) =\\[7pt]
- \hat{d} \Big( I_t ( i_{\xi} \hat{d} \omega_1(t)) - I_t ( i_{\xi}
(\partial_t
\omega_1(t) + \hat{d} \eta(t)))\wedge \eta(t) \Big) ,\\[7pt]
\partial_t(\omega_2(t)\wedge\eta(t))=-\hat{d} \omega_3(t) ,\\[7pt]
\partial_t(\omega_3(t)\wedge\eta(t))=\hat{d} \omega_2(t) ,
\end{array} \right.
\end{equation}
are satisfied, then the $\SU(3)$-structure on $M=N\times (a,b)$
given by
\begin{equation}\label{SU(3)-on-NxI-tres}
F=\omega_1(t)+\eta(t)\wedge dt, \quad\quad \Psi = (\omega_2(t) + i
\omega_3(t))\wedge(\eta(t)- i dt),
\end{equation}
is {\rm SKT}.
\end{theorem}

\begin{example}
{\rm Let us consider the Lie algebra with structure equations
$$
\left \{ \begin{array}{l}
d e^j = 0, j = 1,2,3,\\
d e^4 = e^{12},\\
d e^5 = e^{14},
\end{array}
\right.
$$
underlying the $5$-dimensional nilmanifold $N^5$ considered in Example \ref{example} and endowed
with the
$SU(2)$-structure given by \eqref{su2structex}.
It is
straight forward to verify that
$$d(\omega_2\wedge\eta)=d(\omega_3\wedge\eta)=d(\omega_1\wedge\omega_1)=0.
$$
Let us evolve the previous SU(2)-structure in the following way:
$$
\begin{array}{lll}
\omega_1(t)&=& -e^{14}-e^{53},\\[5pt]
\omega_2(t)&=&-\left(1+\frac32t\right)^{1/3}\,e^{15} - \left(1+\frac32t\right)^{-1/3}\,e^{34},\\[5pt]
\omega_3(t)&=&-\left(1+\frac32t\right)^{1/3}\,e^{13}-\left(1+\frac32t\right)^{-1/3}\,e^{45},\\[5pt]
\eta(t)&=&\left(1+\frac32t\right)^{1/3}\,e^{2},
\end{array}
$$ where $t\in(-2/3,\infty)$.

It is immediate to observe that the family $(\omega_1(t),\,\omega_2(t),\,\omega_3(t),\,\eta(t))$ verifies equations \eqref{cond-hypo-type} and the two last equations in \eqref{evolution-SKT-bis} for any $t\in(-2/3,\infty)$.  Moreover, it verifies the following conditions:
$$\partial_t\omega_1(t)=0,\quad \hat d(\eta(t))=0,\quad i_{\xi}\left(\hat d(\omega_1(t))\right)=0,\quad \partial_t\left(I_t(\hat d \omega_1(t))\right)=0,$$ which implies that the evolution equations \eqref{evolution-SKT} are also satisfied.

On the product $N^5\times\R$ let us consider the local basis of $1$-forms given by
$$\begin{array}{l}\beta^1=\left(1+\frac32t\right)^{1/3}\,e^{1},\quad \beta^2=\left(1+\frac32t\right)^{-1/3}\,e^{4},\quad \beta^3=e^5,\quad \beta^4=e^3,\\[4pt] \beta^5=\left(1+\frac32t\right)^{1/3}\,e^{2},\quad \beta^6=dt.\end{array}$$

The structure equations are:
$$\begin{cases}\begin{array}{l}
d\beta^1\,=\,-\frac12\,\left(1+\frac32t\right)^{-1}\,\beta^{16},\\[5pt]
d\beta^2\,=\,\left(1+\frac32t\right)^{-1}\,\left(\beta^{15}+\frac12\,\beta^{26}\right),\\[5pt]
d\beta^3\,=\,\beta^{12},\\[5pt]
d\beta^4\,=\,0,\\[5pt]
d\beta^5\,=\,-\frac12\,\left(1+\frac32t\right)^{-1}\,\beta^{56},\\[5pt]
d\beta^6\,=\,0.
\end{array}\end{cases}$$
$J$ is given locally by $J\beta^1=-\beta^2,\quad
J\beta^3=-\beta^4,\quad J\beta^5=\beta^6.$  The fundamental form
$F=-\beta^{12}-\beta^{34}+\beta^{56}$ verifies that $d(JdF)=0$ and
the $(3,0)$-form
$\Psi=(\beta^1+i\,\beta^2)\wedge(\beta^3+i\,\beta^4)\wedge(\beta^5-i\,\beta^6)$
is closed.  Therefore, $(F,\Psi)$ is a local SKT SU(3)-structure on
$N^5\times\R$.

 }
\end{example}

\medskip

\smallskip

\begin{remark}
{\rm A Hermitian structure $(J,h)$ on a $6$-dimensional manifold
$M^6$ is called {\em balanced} if $F \wedge F$ is closed, $F$ being
the associated fundamental $2$-form. In~\cite{FTUV} it was
introduced the notion of balanced SU(2)-structures on
$5$-dimensional manifolds, together with appropriate evolution
equations whose solution gives rise to a balanced SU(3)-structure in
six dimensions.

If $M^6$ is compact, then a balanced structure cannot be SKT (see
for instance~\cite{FPS}).

The   SU(2)-structure \eqref{su2structex} on the previous example is
also balanced and it gives rise to a  balanced metric on the product
of $N^5$  with a open interval (see (11) in~\cite{FTUV}). However
one can check directly that this solution is not SKT.}
\end{remark}

Notice that if $G$ is the nilpotent  Lie group  underlying $N^5$, the product $G \times \R$ has no left-invariant SKT structures and
it does not admit any left-invariant complex structures; however we
find a local  SKT SU(3)-structure on it.

\section{HKT structures}

In this section we will find conditions for which an $S^1$-bundle over a $(4n + 3)$-dimensional manifold endowed with three almost contact metric structures is hyper-K\"ahler with torsion (HKT for short). We recall that a $4n$-dimensional  hyper-Hermitian manifold $(M^{4n}, J_1, J_2, J_3, h)$ is a hypercomplex manifold $(M^{4n}, J_1, J_2, J_3)$ endowed with a Riemannian metric $h$ which is compatible with the complex structures $J_r$, $r =1,2,3$, i.e. such  that
$$
h(J_r X, J_r Y) = h (X, Y),
$$
for any $r = 1,2,3$ and any vector fields $X, Y$ on $M^{4n}$.

A  hyper-Hermitian manifold $(M^{4n}, J_1, J_2, J_3, h)$  is called HKT if and only if
\begin{equation} \label{HKTonforms}
J_1dF_1=J_2dF_2=J_3dF_3,
\end{equation}
where $F_r$ denotes the fundamental $2$-form associated to the Hermitian structure $(J_r, h)$ (see \cite{GP}).

Let us consider a $(4n + 3)$-dimensional manifold  $N^{4n+3}$
endowed with three almost contact metric structures $(I_r,\,\xi_r,
\,\eta_r, g)$,  $r=1,2,3$,  such that
\begin{equation} \label{cond-3metric}
\begin{array}{l}
I_k=I_iI_j - \eta_j\otimes\xi_i=-I_jI_i + \eta_i\otimes\xi_j,\\[4pt]
\xi_k=I_i\xi_j=-I_j\xi_i,\quad \eta_k=\eta_iI_j=-\eta_jI_i.
\end{array}
\end{equation}
By applying  Theorem \ref{non-t-circle-bundle}
we can construct    hyper-Hermitian structures on   $S^1$-bundles  over $N^{4n + 3}$ and study when they are strong  HKT.

\begin{theorem}\label{HKTcircle-bundle}
Let $N^{4n + 3}$ be a $(4n+3)$-dimensional manifold with three
normal almost contact metric structures $(I_r,\,\xi_r, \,\eta_r,
g)$, $r=1,2,3$, satisfying \eqref{cond-3metric}, and let $\Omega$ be
a closed $2$-form on $N^{4n+3}$ which represents an integral
cohomology class  and which is $I_r$-invariant for every $r =
1,2,3$.  Consider the circle bundle $S^1 \hookrightarrow P \to N^{4n
+ 3}$ with connection $1$-form $\theta$ whose curvature form is
$d\theta = \pi^*(\Omega)$, where $\pi: P \to N$ is the projection.
Then, the hyper-Hermitian structure $(J_1, J_2, J_3, h)$ on $P$,
defined by~\eqref{acxS1}  and~\eqref{acxS2-1}, is {\rm HKT} if and
only  if
\begin{equation} \label{HKTS"1bundleconditions}
\begin{array}{l}
\pi^* (I_1 (d \omega_1))-  \pi^* (d \eta_1) \wedge \pi^* \eta_1 =\pi^* (I_2 (d \omega_2)) - \pi^* (d \eta_2) \wedge \pi^* \eta_2\\[4pt] = \pi^* (I_3 (d \omega_3))- \pi^* (d \eta_3) \wedge \pi^* \eta_3,\\[9pt]
\pi^* (I_1 (i_{\xi_1} d \omega_1))= \pi^* (I_2 (i_{\xi_2} d \omega_2))
= \pi^* (I_3 (i_{\xi_3} d \omega_3)),
\end{array}
\end{equation}
 where $\omega_r$ denotes the  fundamental form of the almost contact structure  $(I_r,\,\xi_r,
\,\eta_r, g)$.
 Moreover, the {\rm HKT} structure is strong if and only if
  \begin{equation} \label{HKTstrongS2-2}
\begin{array} {l}
d(\pi^* (I_r(i_{\xi_r} d \omega_r))) =0,\\
d(\pi^* (I_r(d \omega_r)- d \eta_r \wedge \eta_r)) = \left(\pi^* (- I_r(i_{\xi_r} d \omega_r))+
\pi^* \Omega\right)\wedge  \pi^* \Omega,
\end{array}
\end{equation}
for every $r = 1,2,3$.
\end{theorem}

\begin{proof}
 The almost hyper-Hermitian structure $(J_1, J_2, J_3, h)$ on $P$,
 defined by~\eqref{acxS1} and~\eqref{acxS2-1},
 is  hyper-Hermitian if and only  $(I_r,\,\xi_r,
\,\eta_r, g)$  is normal and   $d\theta$ is $J_r$-invariant   for
every $r = 1,2,3$. The  HKT condition is equivalent
to~\eqref{HKTonforms}.
 By~\eqref{expressionJdF}  we  have
$$
J_r dF_r =  \pi^*(I_r(d \omega_r)) + \pi^*(I_r (i_{\xi_r} d \omega_r))\wedge
\theta - \pi^*(d \eta_r) \wedge \pi^* \eta_r  - \theta \wedge d  \theta,
$$
where $F_r$ is the fundamental $2$-form of $(J_r, h)$.
Therefore, the condition \eqref{HKTonforms} is satisfied if and only if  \eqref{HKTS"1bundleconditions}
holds.
Finally, $J_r dF_r $ are closed forms if and only if  \eqref{HKTstrongS2-2}
holds.
\end{proof}

Consider  on $N^{4n + 3} \times \R$ the almost Hermitian structures
$(J_r, F_r, h)$ defined by
\begin{equation}
\label{productfundformHKT} h = g + (dt)^2, \quad F_r=\omega_r +
\eta_r\wedge dt,
\end{equation}
 and
\begin{equation} \label{Jrproduct}
J_r(\eta_r)=dt,\quad J_r(X)=I_r(X),\,X\in {\mbox {Ker}} \, \eta_r.
\end{equation}

Moreover, by \eqref{cond-3metric} we have:
$$
\begin{array}{l}
J_1 J_2 = J_3 = - J_2 J_1,\\[4pt]
J_1\eta_2=I_1\eta_2=-\eta_3,\quad J_2\eta_3=I_2\eta_3=-\eta_1,\quad
J_3\eta_1=I_3\eta_1=-\eta_2.
\end{array}$$

Therefore $(J_r,  F_r, h)$, $r = 1,2,3$,  is a  hyper-Hermitian
structure on  $N^{4n + 3} \times \R$ if and only if the structures
$(I_r,\,\xi_r, \,\eta_r)$ for $r=1,2,3$ are normal.

\begin{corollary} \label{HKTproduct}  Let $N^{4n + 3}$ be a $(4n + 3)$-dimensional  manifold endowed with three  normal almost contact metric structures
$(I_r,\,\xi_r, \,\eta_r, g)$, $r=1,2,3$.  Consider  on the  product
$N^{4n + 3} \times \R$  the    hyper-Hermitian structure  $(J_1,
J_2, J_3, h)$ defined  by~\eqref{productfundformHKT}
and~\eqref{Jrproduct}. Then, $(J_1, J_2, J_3, h)$ is {\rm HKT} if
and only  if 
$$
\begin{array}{c}
I_1 (d \omega_1)- d \eta_1 \wedge  \eta_1 = I_2 (d \omega_2) -  d \eta_2 \wedge  \eta_2 = I_3 (d \omega_3) -  d \eta_3 \wedge  \eta_3,\\[9pt]
I_1 (i_{\xi_1} d \omega_1)=  I_2 (i_{\xi_2} d \omega_2)
=  I_3 (i_{\xi_3} d \omega_3).
\end{array}
$$
The HKT structure is strong if and only if 
$$
d(I_r(i_{\xi_r} d \omega_r)) =0, \quad
d(I_r(d \omega_r)- d \eta_r \wedge \eta_r) = 0
$$
for every $r = 1,2,3$. 

Moreover, if  $(J_1, J_2, J_3, h)$ is such that   $$
d \eta_1 \wedge \eta_1 = d \eta_2 \wedge \eta_2 = d \eta_3 \wedge
\eta_3,
$$
 and one of the following conditions:
\begin{enumerate}
\item[(a)] $d\omega_r=0$ for any $r=1,2,3$, i.e. $(I_r,\,\xi_r,\,\eta_r)$ is  quasi-Sasakian for any $r=1,2,3$
or
\item[(b)] $d\omega_i\wedge\eta_{j}\wedge\eta_{k}\neq 0$, where $(i,j, k)$ is a permutation of $(1,2,3)$,  and
$$
I_1(d\omega_1)=I_2(d\omega_2)=I_3(d\omega_3),\quad
I_1(i_{\xi_1}d\omega_1)=I_2(i_{\xi_2}d\omega_2)=I_3(i_{\xi_3}d\omega_3),
$$
\end{enumerate}
is satisfied, 
then $(J_1, J_2, J_3, h)$ is {\rm HKT}.   In the case $(a)$ the {\rm HKT} structure is strong. In the case $(b)$ the
{\rm HKT} structure is strong if and only if
$$d\left(I_1(d\omega_1)\right)=d\left(I_1(i_{\xi_1}d\omega_1)\right)=0.$$ 
\end{corollary}

\begin{proof}
By Theorem  \ref{HKTcircle-bundle} the hyper-Hermitian structure $(J_r,  F_r, h)$, $r = 1,2,3$,  is HKT if and only if
\begin{equation}\label{HKTcondproduct}
\begin{array}{l}
I_1 (d \omega_1)-  d \eta_1 \wedge   \eta_1= I_2 (d \omega_2) -  d \eta_2 \wedge   \eta_2 = I_3 (d \omega_3) -  d \eta_3 \wedge   \eta_3,\\[6pt]
I_1 (i_{\xi_1} d \omega_1)  = I_2 (i_{\xi_2} d \omega_2)
= I_3 (i_{\xi_3} d \omega_3).
\end{array}
\end{equation}

Let us express locally \begin{equation} \label{exprdomegar}
d\omega_r=\alpha_r + \sum_{i=1}^3 \beta_i^r\wedge\eta_i+
\sum_{i<j=1}^3 \gamma^r_{ij}\wedge\eta_{i}\wedge\eta_j+
\rho_r\,\eta_{1}\wedge\eta_2\wedge\eta_3,
\end{equation}
where $\alpha_r,\,\beta_i^r$ and $\gamma^r_{ij}$ are 3-forms,
2-forms and 1-forms respectively in $\bigcap_{i = 1}^3\,{\mbox {Ker}}\,\eta_i$ and
$\rho_r$ are smooth functions.

By using the normality of the three almost contact metric
structures, and then that $i_{\xi_r}d\eta_r=0$ and
$I_r(d\eta_r)=d\eta_r$,  we can write locally:
\begin{equation} \label{exprdetar}
\begin{array}{lll}
d\eta_1&=&A_1 + B_1\wedge\eta_2 - I_1 B_1\wedge\eta_3 +
C_1\,\eta_{2}\wedge\eta_3,\\[5pt]
d\eta_2&=&A_2 + B_2\wedge\eta_1 + I_2 B_2\wedge\eta_3 +
C_2\,\eta_{1}\wedge\eta_3,\\[5pt]
d\eta_3&=&A_3 + B_3\wedge\eta_1 - I_3 B_3\wedge\eta_2 +
C_3\,\eta_{1}\wedge\eta_2,
\end{array}
\end{equation}
where $I_rA_r=A_r$.  $A_r$ and $B_r$ are 2-forms and 1-forms
respectively in $\bigcap_{i = 1}^3\,{\mbox {Ker}}\,\eta_i$ and $C_r$ are smooth
functions.

We have
 $$
 J_r(dF_r)=J_r(d\omega_r) + J_r(d\eta_r\wedge dt) = J_r(d\omega_r) -d\eta_r\wedge\eta_r.
 $$
Therefore, by using \eqref{exprdomegar} and \eqref{exprdetar}, we
obtain
 \begin{eqnarray*}
J_1(dF_1)&=&I_1\alpha_1 + I_1\beta^1_1\wedge dt- A_1\wedge\eta_1-
I_1\beta^1_3\wedge\eta_2 - I_1\beta^1_2\wedge\eta_3\\[5pt]&&  - I_1\gamma^1_{13}\wedge\eta_2\wedge dt
+I_1\gamma^1_{12}\wedge\eta_3\wedge dt +
B_1\wedge\eta_{1}\wedge\eta_2 -
I_1B_1\wedge\eta_{1}\wedge\eta_3\\[5pt] && + I_1\gamma^1_{23}\wedge\eta_{2}\wedge\eta_3 + \rho_1\,\eta_{2}\wedge\eta_3\wedge dt -
C_{1}\,\eta_{1}\wedge\eta_2\wedge\eta_3,
\end{eqnarray*}
\begin{eqnarray*}
J_2(dF_2)&=&I_2\alpha_2 + I_2\beta^2_2\wedge dt-
I_2\beta^2_3\wedge\eta_1-
A_2\wedge\eta_2 + I_2\beta^2_1\wedge\eta_3\\[5pt]&&  + I_2\gamma^2_{23}\wedge\eta_1\wedge dt
+I_2\gamma^2_{12}\wedge\eta_3\wedge dt -
B_2\wedge\eta_{1}\wedge\eta_2 +
I_2\gamma^2_{13}\wedge\eta_{1}\wedge\eta_3\\[5pt] && + I_2B_2\wedge\eta_{2}\wedge\eta_3 - \rho_2\,\eta_{1}\wedge\eta_3\wedge dt +
C_{2}\,\eta_{1}\wedge\eta_2\wedge\eta_3,
\end{eqnarray*}
\begin{eqnarray*}
J_3(dF_3)&=&I_3\alpha_3 + I_3\beta^3_3\wedge dt+
I_3\beta^3_2\wedge\eta_1- I_3\beta^3_1\wedge\eta_2-
A_3\wedge\eta_3 \\[5pt]&&  + I_3\gamma^3_{23}\wedge\eta_1\wedge dt
-I_3\gamma^3_{13}\wedge\eta_2\wedge dt
+I_3\gamma^3_{12}\wedge\eta_{1}\wedge\eta_2
-B_3\wedge\eta_{1}\wedge\eta_3
\\[5pt] && + I_3B_3\wedge\eta_{2}\wedge\eta_3 + \rho_3\,\eta_{1}\wedge\eta_2\wedge dt - C_{3}\,\eta_{1}\wedge\eta_2\wedge\eta_3.
\end{eqnarray*}

The conditions  \eqref{HKTcondproduct} are satisfied if and
only if
\begin{equation}
\begin{array}{l} \label{genconHKT}
\gamma^1_{12}=\gamma^1_{13}=\gamma^2_{12}=\gamma^2_{23}=\gamma^3_{13}=\gamma^3_{23}=0,\quad
\rho_r=0,\quad
C_{1}=-C_{2}=C_{3},\\[5pt]
I_1\alpha_1=I_2\alpha_2=I_3\alpha_3,\quad
I_1\beta^1_1=I_2\beta^2_2=I_3\beta^3_3,\\[5pt]
A_1=I_2\beta^2_3=-I_3\beta^3_2,\quad A_2=-I_1\beta^1_3=I_3\beta^3_1,\quad A_3=I_1\beta^1_2=-I_2\beta^2_1,\\[5pt]
B_1=-B_2=I_3\gamma^3_{12},\quad -I_1B_1=-B_3=I_2\gamma^2_{13},\quad
I_2B_2=I_3B_3=I_1\gamma^1_{23}.\end{array}\end{equation}
\vspace{.1cm}

Since $I_rA_r=A_r$ we obtain that the coefficients $\beta^r_i$ for $r\neq i=1,2,3$ must satisfy the following conditions:
$$I_i\left(\beta^i_j - I_k\beta^i_j\right)=0,\quad\forall i,j,k=1,2,3,\quad i\neq j,\,\, j\neq k,\,\, k\neq i.$$ The last three equations in \eqref{genconHKT} are
satisfied if and only if
$\gamma^1_{23}=\gamma^2_{13}=\gamma^3_{12}=0$.

Thus, finally, we obtain:
\begin{equation}
\begin{array}{l} \label{finalconHKT}
d\omega_r=\alpha_r + \sum_{i=1}^3\beta^r_i \wedge \eta_i,\quad d\eta_i=A_i +\lambda\, \eta_{j}\wedge\eta_{k},\\[6pt]
I_i\left(\beta^i_j - I_k\beta^i_j\right)=0,\quad\forall i,j,k=1,2,3,\quad i\neq j,\,\, j\neq k,\,\, k\neq i,\\[6pt]
I_1\alpha_1=I_2\alpha_2=I_3\alpha_3,\\[6pt]
A_1=I_2\beta^2_3=-I_3\beta^3_2,\quad A_2=-I_1\beta^1_3=I_3\beta^3_1,\quad A_3=I_1\beta^1_2=-I_2\beta^2_1.
 \end{array}
 \end{equation} 
  
  for any even permutation of $(1,2,3)$.

Now, the expression for $d(J_1dF_1)$ is the following:
$$
\begin{array}{lll}
d(J_1dF_1)
&=&d\left(I_1(d\omega_1) + I_1({i_{\xi}}_1d\omega_1)\wedge dt\right) - d\left((d\eta_1)\wedge\eta_1\right)\\[5pt]
&=&d\left(I_1(d\omega_1)\right) + d\left(I_1({i_{\xi}}_1d\omega_1)\right)\wedge dt - d\eta_1\wedge d\eta_1\\[5pt]
&=&d\left(I_1(d\omega_1)-d\eta_1\wedge\eta_1\right) + d\left(I_1({i_{\xi}}_1d\omega_1)\right)\wedge dt,\end{array}
$$
and thus the HKT structure is strong if and only if $$d (I_1
(d\omega_1)-d\eta_1\wedge\eta_1) =0,\quad \text{and}\quad d(I_1({i_{\xi}}_1d\omega_1))=0.$$ To prove the last part of the corollary  it is sufficient to consider coefficients $\beta^i_r=0$ if $r\neq i$ in expression \eqref{exprdomegar}.

 \end{proof}

\begin{example} {\rm  Consider the $7$-dimensional Lie group $G=$ SU(2) $\ltimes \R^4$ with structure equations
$$\begin{cases}
\begin{array} {l}
d e^1 = - \frac{1}{2} e^{25} - \frac{1}{2} e^{36} - \frac{1}{2} e^{47},\\[4pt]
d e^2 =  \frac{1}{2} e^{15}+ \frac{1}{2} e^{37} - \frac{1}{2} e^{46},\\[4pt]
d e^3 = \frac{1}{2} e^{16}- \frac{1}{2} e^{27} + \frac{1}{2} e^{45},\\[4pt]
d e^4 =  \frac{1}{2} e^{17}+ \frac{1}{2} e^{26} - \frac{1}{2} e^{35},\\[4pt]
d e^5 = e^{67},\\
d e^6 = - e^{57},\\
d e^7 = e^{56}.
\end{array}\end{cases}
$$
By~\cite{FT} $G$ admits a compact quotient $M^7 = \Gamma \backslash
G$  by a uniform discrete subgroup  $\Gamma$ and it is endowed with
a weakly generalized $G_2$-structure. Moreover, by \cite{BF} $M^7
\times S^1$ admits a strong HKT structure.  We can show that $M^7$
has three normal almost contact metric structures  $(I_r,\,\xi_r,
\,\eta_r, g)$ for $r=1,2,3$ given by
$$
\begin{array}{l}
I_1 e^1 =  e^2, \quad  I_1 e^3 = e^4, \quad  I_1 e^5 = e^6, \quad \eta_1 = e^7 ,\\[4pt]
I_2 e^1 = e^3, \quad  I_2 e^2 =  -e^4, \quad  I_2e^5 = - e^7, \quad \eta_2 = e^6,\\[4pt]
I_3 e^1 =  e^4, \quad  I_3 e^2 =  e^3, \quad  I_3 e^6 =  e^7, \quad
\eta_3 = e^5,
\end{array}
$$
 satisfying the conditions $(a)$ of  Corollary~\ref{HKTproduct}.

}
\end{example}

\bigskip

\bigskip

\medskip

\noindent {\bf Acknowledgments.} This work has been partially
supported through Project MICINN (Spain) MTM2008-06540-C02-01/02,
Project MIUR ``Riemannian Metrics and Differentiable Manifolds"  and
by GNSAGA of INdAM.

\smallskip

{\small


\begin{thebibliography}{33}
\bibitem{AFV} A. Andrada, A. Fino, L. Vezzoni, A class of  Sasakian $5$-manifolds,  preprint math. DG/0807.1800 , to appear in \emph{Transformation Groups}.

\bibitem{BDV} M. L. Barberis, I. Dotti, M. Verbitsky, Canonical bundles of complex nilmanifolds, with applications to hypercomplex geometry, \emph{Math. Res. Lett.} {\bf 16} (2009), 331--347.


\bibitem{BF} M. L. Barberis, A. Fino, New strong HKT manifolds arising from quaternionic representations,  preprint math.DG/0805.2335, to appear in \emph{Math. Z.}.

\bibitem{Bi} J.M. Bismut, A local index theorem for non-K\"ahler manifolds, \emph{Math. Ann.} {\bf 284} (1989), 681--699.

\bibitem{Blair} D. L. Blair, The theory of quasi-Sasakian structures, \emph{J. Differ. Geom.} {\bf  1} (1967), 331-345.

\bibitem{Blair2} D. L. Blair, \emph{Riemannian geometry of contact and symplectic manifolds},  Progress in Mathematics, {\bf 203},  Birkh\" auser Boston, Inc., Boston, MA, 2002.

\bibitem{BG} C. P. Boyer, K. Galicki, $3$-Sasakian manifolds, in:  \emph{Surveys in Differential Geometry: Essays on Einstein Manifolds}, Surveys in Differential Geometry, Vol.
VI, Int. Press, Boston, MA, 1999, pp. 123--184.

\bibitem{CS} D. Conti, S. Salamon, Generalized Killing spinors in dimension 5,
\emph{Trans. Amer. Math. Soc.} {\bf 359} (2007), 5319--5343.

\bibitem{Eg} N. Egidi, Special metrics on compact complex manifolds,  \emph{Differential Geom. Appl.} {\bf 14} (2001),  217--234.

\bibitem{FTUV} M. Fern\'{a}ndez, A. Tomassini, L. Ugarte, R. Villacampa,
Balanced Hermitian metrics from SU(2)-structures, \emph{J. Math.
Phys.} {\bf 50} (2009), 033507.

\bibitem{FG} A. Fino, G. Grantcharov, Properties of manifolds with skew-symmetric torsion and special holonomy, \emph{Adv. Math.} {\bf 189} (2004), 439--450.


\bibitem{FPS} A. Fino, M. Parton, S. Salamon, Families of strong KT structures
in six dimensions, \emph{Comment. Math. Helv.} {\bf 79} (2004),
317--340.

\bibitem{FT}  A. Fino, A. Tomassini, Generalized $G\sb 2$-manifolds and ${\rm SU}(3)$-structures, \emph{Internat. J. Math.} {\bf 19} (2008),  1147--1165.

\bibitem{FT2} A. Fino,  A. Tomassini,  Blow-ups and resolutions of strong K\"ahler with torsion metrics, \emph{Adv. Math.} {\bf  221} (2009),  914--935.



\bibitem{FI2} Th. Friedrich, S. Ivanov,  Parallel spinors and connections with skew-symmetric torsion in string theory,  \emph{Asian J. Math.}  {\bf 6}  (2002), no. 2, 303--335.

\bibitem{GP} G. Grantcharov, Y.S. Poon, Geometry of hyper-K\" ahler connections with torsion, \emph{Comm. Math. Phys.} {\bf 213} (2000),  19--37.


\bibitem{GGP}  D. Grantcharov, G. Grantcharov, Y.S. Poon, Calabi-Yau connections with torsion on toric
bundles,  \emph{J. Differ. Geom.} {\bf 78} (2008), 13--32.


\bibitem{GHR} S.J. Gates, C.M. Hull, M. Ro\v cek, Twisted multiplets and new supersymmetric non-linear
$\sigma$-models, \emph{Nuclear Phys. B}  {\bf 248} (1984), 157--186.

\bibitem{Ga} P. Gauduchon,  La $1$-forme de torsion d'une vari\'et\'e hermitienne compacte, \emph{Math. Ann.} {\bf 267}
(1984), 495--518.

\bibitem{H} N.J. Hitchin, Stable forms and special metrics. In: Fern\'andez, M., Wolf J. (eds.),
Global differential geometry: the mathematical legacy of Alfred Gray
(Bilbao, 2000), Contemp. Math. {\bf 288}, Amer. Math. Soc.,
Providence, RI, 2001, 70--89.

\bibitem{HP} P. S.  Howe,  G. Papadopoulos,  Twistor spaces for hyper-K\" ahler manifolds with torsion, \emph{Phys. Lett. B} {\bf 379} (1996),  80--86.

\bibitem{IP}   S. Ivanov and G. Papadopoulos, Vanishing theorems and string backgrounds, \emph{Classical Quantum Gravity}  {\bf 18} (2001), 1089--1110.

\bibitem{Kob} S.Kobayashi, Principal fibre bundles with the
1-dimensional toroidal group, \emph{Tohoku Math. J.} {\bf 8}
(1956), 29--45.

\bibitem{Og} Y. Ogawa, Some properties on manifolds with almost contact structures,
\emph{Tohoku Math. J.} {\bf 15} (1963), 148--161.

\bibitem{SH} S. Sasaki, Y. Hatekeyama, On differentiable manifolds with certain structures which are closely related to almost contact structure II, \emph{Tohoku Math. J.} {\bf 13} (1961), 281--294.

\bibitem{Str} A. Strominger, Superstrings with torsion, \emph{Nuclear Phys. B} {\bf 274} (1986), 253--284.

\bibitem{Sw} A. Swann: Twisting Hermitian and hypercomplex geometries,  Duke Math. J. {\bf 155} (2010),403--431.

\bibitem{U} L. Ugarte, Hermitian structures on six dimensional nilmanifolds, \emph{Transform. Groups} {\bf 12} (2007), 175--202.

\end{thebibliography}
\end{document}